\documentclass[12pt]{wart}
\usepackage{xspace,amssymb,amsfonts,euscript,amsthm,amsmath}
\usepackage{graphicx,ifpdf}
\ifpdf \usepackage{epstopdf}
 \usepackage{pdfsync}\fi
\usepackage{palatino}
\usepackage[USenglish]{babel}

 \title[Light leaves and Lusztig's conjecture]
 {Light leaves and Lusztig's conjecture}
\author{\textsc{Nicolas Libedinsky}}
\address{Universidad de Chile,\\
Facultad de Ciencias,\\
Casilla 653, Santiago, Chile}
\email{nlibedinsky@gmail.com}

\input xy
\xyoption {all}

  \newcommand{\nc}{\newcommand}
  \newcommand{\renc}{\renewcommand} 

\usepackage[latin1]{inputenc}

\nc{\E}{\mathbb{E}}

\nc{\Q}{\mathbb{Q}}

\nc{\triright}{\stackrel{[1]}{\to}}

\def\to{\rightarrow}

\newcommand{\CZ}{{\mathcal Z}}

\newcommand{\hCZ}{{\CZ}^{\circ}}
\nc{\Br}{\mathcal{B}}
\nc{\id}{id}
\nc{\HotRR}{{}_R\mathcal{K}_R}
\nc{\HotR}{\mathcal{K}_R}
\nc{\excise}[1]{}
\nc{\defect}{\text{df}}
\nc{\h}[1]{\underline{H}_{#1}}
\nc{\Z}{\mathbb{Z}}
\nc{\R}{\mathbb{R}}
\nc{\C}{\mathbb{C}}
\renc{\P}{\mathbb{P}}
\renc{\O}{\mathcal{O}}
\nc{\N}{\mathbb{N}}

\nc{\F}{\mathcal{F}}
\nc{\G}{\mathcal{G}}

\nc{\nilp}{\mathcal{N}}

\nc{\Ga}{\mathbb{G}_a} % additive group
\nc{\Gm}{\mathbb{G}_m} % multiplicative group

\nc{\Loc}{\mathcal{L}}

\nc{\A}{\mathbb{A}} % affine space

\nc{\IC}{\mathbf{IC}}
\nc{\D}{\mathbb{D}}

\DeclareMathOperator{\supp}{supp}

\newtheorem{term}{Terminology}

  \newtheorem{defi}{Definition}
  \newtheorem{thm}{Theorem}[section]
  \newtheorem{lem}[thm]{Lemma}
  \newtheorem{deth}[thm]{Definition/Theorem}
  \newtheorem{prop}[thm]{Proposition}
  \newtheorem{cor}[thm]{Corollary}
  \newtheorem{conj}[thm]{Conjecture}
  \newtheorem{nota}[thm]{Notation}

  \theoremstyle{remark}
  \newtheorem{remark}{Remark}

\newcommand{\BMod}[2]{{#1}\textrm{-Mod-}#2}
\newcommand{\Gr}{\textrm{Gr}}

\begin{document}
\begin{abstract}
We define a map $F$ with domain a certain subset of the set of 
\emph{light leaves}  (combinatorial objects introduced by the author in a previous paper) and range the set of prime numbers. We prove the following property of $F$:  if the image $p=F(l)$ of some light leaf $l$ under $F$ is bigger than the Coxeter number of the corresponding Weyl group, then there is a counterexample to Lusztig's conjecture in characteristic $p$. 
We also introduce the  ``double leaves basis" which is an amelioration of the light leaves basis  that has already found interesting applications. In particular it forms a cellular basis of Soergel bimodules that allows us to  produce an algorithm to find ``the bad primes" for Lusztig's conjecture. 

%We give a combinatorial algorithm to find, for any given Weyl or affine Weyl group, the set of primes for which Soergel's conjecture holds. This conjecture for Weyl groups is equivalent to a part of Lusztig's conjecture and for affine Weyl groups implies (and is probably equivalent to) the full Lusztig conjecture. %(for the simple rational characters of  reductive algebraic groups) 
%The aforementioned algorithm is based on the \emph{light leaves basis}, a combinatorial basis introduced by the author in a previous paper for the Hom spaces between two Bott-Samelson-Soergel bimodules. Some months after this work was posted on the arXiv, Geordie Williamson found counterexamples to Lusztig's conjecture using a slightly different form of this algorithm (discovered independently). Surprisingly, these counterexamples seem to grow exponentially in the Coxeter number. 
\end{abstract}
\maketitle

\section{Introduction}

\subsection{Lusztig's Conjecture} Let $n\geq 2$ be an integer and $p$ a prime number. Consider the following question.

 \textbf{Q}: \emph{What are the characters of the irreducible rational representations of} $GL_n(\overline{{\mathbb{F}}}_p)$ \emph{over} $\overline{\mathbb{F}}_p?$

 This natural question gained interest in the year 1963 when Steinberg  proved that  all the irreducible  representations of  the  finite groups  $GL_n({\mathbb{F}}_q)$ (with $q$ a power of $p$) could be obtained from the irreducible representations of $GL_n(\overline{{\mathbb{F}}}_p)$ by restriction (see \cite{St}).

 The groups  $GL_n({\mathbb{F}}_q)$ are examples of finite groups of Lie type, i.e. groups $G({\mathbb{F}}_q)$ of rational points of a reductive linear algebraic group $G$  defined over  ${\mathbb{F}}_q$. In the mid-seventies it  became likely that (minor modifications of) the finite groups of Lie type, together with the cyclic and alternating groups, would give all the infinite families of finite simple groups; this gave a new impulse to the study of \textbf{Q}.

A big breakthrough came in 1979 when Lusztig \cite{Lu1} gave a conjectural answer for \textbf{Q} when $p>n,$ and more  generally, he conjectured  a formula for the characters of the irreducible rational representations of any reductive algebraic group $G$ over $\overline{\mathbb{F}}_p$ when $p$ is bigger than $h$, the Coxeter number of the Weyl group associated to $G$. This is known as \emph{Lusztig's Conjecture}  \footnote{Actually   Lusztig made his conjecture only for those irreducible modules with high weights in the "Jantzen region", a finite collection of weights containing all "restricted" weights for $p\geq 2h-3,$ and thus giving formulas for any high weight by previous work of Steinberg. This conjecture was later generalized by Kato \cite{Ka} to all restricted weights  for $p\geq h$. This generalization is what we call "Lusztig's conjecture". In the literature it is either called "Lusztig's conjecture", "Kato's conjecture" or "Kato's extension of the Lusztig conjecture". } for algebraic groups. Geordie Williamson \cite{Wi3} proved  in September 2013 that this conjecture is false using the light leaves approach. 

\subsection{Main approaches toward Lusztig's Conjecture}
We include this section  for historical reasons, it is not needed in the rest of the text.

There have been several approaches to Lusztig's conjecture throughout the years. Let us recall some of them.

One approach was due to E. Cline, B. Parshall and  L. Scott. They wrote many papers, starting in the late eighties, related to Lusztig's conjecture  and  to the algebraic understanding of the geometry involved in Kazhdan-Lusztig theory (see \cite{Sc} for an overview). Let $G$ be a reductive algebraic group and $A$  the quasi-hereditary algebra associated to the category of $G$-modules whose composition factors have regular high weights in a poset $\Gamma$ taken to be the Jantzen region. They proved \cite[remark 2.3.5]{CPS} that if $A$ is Koszul (in a suitably strong sense) then Lusztig's conjecture is true.  Parshall and Scott \cite{PS} have recently made a contribution to this approach by showing that it is sometimes possible to transfer the good Lie theoretic properties to $\mathrm{gr}A$, thereby artificially creating much the same situation that would be obtained from an actual Koszul grading on $A$ itself.

%They introduced the important concept of \emph{highest-weight category} \cite{CPS1} as %a first attempt at giving an algebraic version of perverse sheaves. 
%Then they studied algebraically the Frobenius morphism and proved \cite{CPS2} the %Koszulity of algebras associated to perverse sheaves on flag varieties (announced by 
%Beilinson-Ginsberg). 

In 1990 Lusztig (\cite{Lu2}, \cite{Lu3})  himself outlined a program for proving his conjecture when $p$ is  "large enough". This program used deep algebraic geometry. It was fulfilled in several steps by Kashiwara-Tanisaki (\cite{KT1}, \cite{KT2}), Kazhdan-Lusztig (\cite{KL1}, \cite{KL2}, \cite{KL3}, \cite{KL4}) and Andersen-Jantzen-Soergel (\cite{AJS}). They essentially reduced the problem to the one of calculating the local intersection cohomology over the complex numbers of some finite-dimensional Schubert varieties in an affine flag variety.  However this approach seems not to lead to any reasonable bounds on $p$. Thus, apart from the cases $A_1, A_2, A_3, B_2, G_2$, the
characters of the irreducible modules of $G$ remained unknown for a given characteristic.

A more recent trend is R. Bezrukavnikov's geometric approach, in collaboration with I. Mirkovic, D. Rumynin and S. Arkhipov. This program is implicit in Bezrukavnikov's ICM address \cite{Be}. In  2010  Bezrukavnikov and Mircovi\'c \cite{BM} proved a group of conjectures of Lusztig \cite{Lu4} (that at some point were known as ``Lusztig's Hope") relating the canonical basis in the homology of a Springer fiber
to modular representations of semi-simple Lie algebras. Their proof is based on the works \cite{BMR1}, \cite{BMR2} and \cite{AB}. Lusztig's conjecture (on representations of algebraic groups in positive characteristic) is essentially equivalent to the particular case of restricted representations, when the relevant Springer fiber is the full flag variety. Thus with this method they gave a new proof of Lusztig's conjecture for $p\gg h.$

In 2010 P. Fiebig \cite{Fi5} improved on these results by constructing, for any given root system $R$, an explicit number $N(R)$  such that
Lusztig's conjecture is true for $p>N(R).$ The number $N(R)$ is enormously big compared with the Coxeter number. To find this number $N(R)$ he introduces a category of sheaves on moment graphs and translates the problem into that language.

\subsection{Soergel's approach}\label{Sa}
  For any Coxeter group $W,$ Soergel \cite{So2}  constructed a polynomial ring $R$ with coefficients in the real numbers $\mathbb{R}.$ Then he constructed a concrete category $\mathcal{B}=\mathcal{B}(W,\mathbb{R})$ of graded bimodules over $R,$  called the category of \emph{Soergel bimodules}.

 Soergel \cite{So2} proved  that $\mathcal{B}$ categorifies the Hecke algebra $\mathcal{H}$ of $W$ in the sense that he constructed a ring isomorphism $$\varepsilon: \mathcal{H}\rightarrow \langle \mathcal{B}\rangle,$$ where $\langle \mathcal{B}\rangle$ denotes the split Grothendieck group of $ \mathcal{B}$. He stated a beautiful conjecture (that we call \emph{Soergel's $0$-conjecture}, even though it has been a theorem since 2013 \cite{EW1}) saying that, if $\{C'_x\}_{x\in W}$ is the Kazhdan-Lusztig basis of the Hecke algebra, then $\varepsilon(C'_x)=\langle B_x \rangle$, where $B_x$ is an indecomposable object of $\mathcal{B}$.

We call this ``Soergel's $0$-conjecture" because it deals with characteristic $0$. There is an analogue of this conjecture in positive characteristic, but the situation there is more complicated.

If $W$ is a Weyl group then everything works as well over a field of characteristic $p$ and Soergel's conjecture in this context will be called \emph{Soergel's $p$-conjecture}.

If $\widehat{W}$ is an affine Weyl group, then there is  a finite subset $\widehat{W}^{\circ}$ of $\widehat{W}$ (for details see Section \ref{P0}) such that all the theory works  in characteristic $p$ if one replaces $W$ by $\widehat{W}^{\circ}$. This version of Soergel's conjecture will be called  \emph{Soergel's affine $p$-conjecture} and it will mean that $\varepsilon(C'_x)=\langle B_x \rangle$ for all elements $x\in \widehat{W}^{\mathrm{res},-}$, where $\widehat{W}^{\mathrm{res},-}$ is some subset of $W^{\circ}.$ The naive generalization of Soergel's conjecture to this context, i.e. to conjecture that this equation is true for every element $x\in \widehat{W}$ is probably not true. We would like to note that  Soergel never conjectured any statements for positive characteristic.

Soergel \cite{So1} himself proved what we have called Soergel's $0$-conjecture   for Weyl groups.  M. H\"{a}rterich \cite{Ha} proved it for affine Weyl groups. Soergel's $0$-conjecture for universal Coxeter groups was proved independently by P. Fiebig  \cite{Fi2} and the author  \cite{Li6}. 

Recently B. Elias and G. Williamson \cite{EW1} gave a striking algebraic proof of  Soergel's $0$-conjecture for any Coxeter system . By previous results of Soergel, their work gave a proof of the longstanding  Kazhdan-Lusztig positivity conjecture for every Coxeter system and an algebraic proof of Kazhdan-Lusztig conjecture. With their method they not only solved these conjectures, but they invented a new mathematical subject: ``algebraic Hodge theory" by looking at Soergel bimodules as if they were the intersection cohomology of some (non-existent) spaces. When $W$ is a Weyl or affine Weyl group, these spaces do exist.

Soergel proved  in  \cite{So3} that  Soergel's $p$-conjecture is equivalent to a part of Lusztig's conjecture (for weights around the
Steinberg weight). We prove in  Section \ref{eq1} something that was in principle known to experts (and implicit in the work of Fiebig) but not written down anywhere to the author's knowledge. We prove that  Soergel's affine $p$-conjecture  is equivalent to Fiebig's conjecture about sheaves on moment graphs. Fiebig proved that his conjecture implies the full Lusztig conjecture. We believe that the converse is also true: Lusztig's conjecture implies Fiebig's conjecture, but we are not yet able to prove it (this would be relevant to find the complete set of counterexamples to Lusztig's conjecture).

 \subsection{Double leaves basis}
 
 Now we explain our approach to Soergel's conjecture. For each reduced expression $\underline{s}$ of an element $x\in W,$ Soergel constructed an explicit  Soergel bimodule $M_{\underline{s}}$  called the \emph{Bott-Samelson} bimodule (see Section \ref{So} for details). For $x \in {W},$ we fix a reduced expression of $x$ and we call $M_x$ the corresponding Bott-Samelson.  In \cite{Li1} we constructed combinatorially, for $x\in W$ a perfect binary tree  (i.e. a tree where every node that is not a leaf has two children, and where all leaves are at the same depth) that we called $\mathbb{T}_x$. The nodes of $\mathbb{T}_x$ are colored by Bott-Samelson bimodules and the edges are colored by morphisms from the corresponding parent to child. The root node of $\mathbb{T}_x$ is colored by $M_x$  and if a leaf is colored by $M_y$ then the composition of all the  morphisms in the path from the root to that leaf gives an element in $\mathrm{Hom}(M_x,M_y)$ that we call a \emph{leaf from $x$ to $y$}. We identify each leaf in $\mathbb{T}_x$ with the corresponding morphism between Bott-Samelson bimodules.
 
 For each leaf $l$  from $x$ to $y$ there is an adjoint leaf $l^a \in \mathrm{Hom}(M_y,M_x)$ (in fact, every morphism between Bott-Samelson bimodules has an adjoint morphism inverting the sense of the arrow).  If $x,y\in W$  and  we have leaves $l_u$ from $x$ to $u$ and $l_v$ from $y$ to $v,$ we define the product $l_v\cdot l_u$ as $l_v^a\circ l_u \in \mathrm{Hom}(M_x,M_y)$ if $u=v,$ and as the empty set if $u\neq v.$

\vspace{0.2cm}

\textbf{Theorem 1} $\ $\emph{ The set $\{l'\cdot l \, |\, l\, \mathrm{ leaf\, of\,  }\,\mathbb{T}_x \, \mathrm{ and\, }l'\, \mathrm{ leaf\, of\,  }\,\mathbb{T}_y\}$ is a basis of $\mathrm{Hom}(M_x,M_y)$ as a left $R$-module.}

\vspace{0.2cm}

We give the proof in section \ref{ac}.
 We call this  the \emph{double leaves basis}  of $\mathrm{Hom}(M_x,M_y)$.  We can see this basis as ``gluing" the tree $\mathbb{T}_x$ with the tree $\mathbb{T}_y$ inverted. %This construction is heavily based on that of the \emph{light leaves basis} in \cite{Li1}.

We can mention three  applications of Theorem 1:
\begin{itemize}
\item The presentation of Soergel's bimodule category by generators and relations in \cite{EW3}. The central theorem  \cite[Theorem 1.1]{EW3} says that the double leaves basis is a basis in the diagrammatic setting.
\item The recent counterexamples to Lusztig's conjecture in \cite{Wi3}. Double leaves are the combinatorial tool to actually calculate these counterexamples (see for example, pages $3$ and $6$ of \cite{Wi3}).
\item The proof of Soergel's $0$-conjecture and Kazhdan-Lusztig conjectures in  \cite{EW1}. In a personal communication G. Williamson told the author ``the original proof of Soergel's conjecture was in the light leaves language, which allowed us to do many calculations, and appears more natural at several steps in the proof. It will appear explicitly in  \cite{EW4}. We could not publish it right away because the results relied on the long papers \cite{EW2} and \cite{EW3} which are still in preparation\footnote{The paper \cite{EW3} is already in the arXiv.}. Hence we decided to make the proof purely algebraic so that it only relied on citable results".
\end{itemize}

 \subsection{The map $F,$ and counterexamples to Lusztig's conjecture}\label{F}

Let $W$ denote a Weyl group. We will work with Soergel bimodules over the field of rational numbers $\mathbb{Q}$.
%of $W$ consisting of the elements that are lesser or equal (in the Bruhat order) than the longest element in the set of antidominant restricted elements. 

Let us fix some $x \in W.$ We define $\mathcal{A}_x$ as the set  of leaves $l$  in $\mathbb{T}_x$ satisfying the following.
\begin{enumerate}
\item $l$ is  degree zero leaf. 
\item If $l$ is a leaf from $x$ to $y$ then $l$ is the only degree zero leaf from $x$ to $y$.
\end{enumerate}

Let us now define the map $F: \mathcal{A}_x\rightarrow \mathbb{Z}.$ If $l\in \mathcal{A}_x$ is a leaf from $x$ to $y,$ the morphism $l\circ l^a$ is a degree zero endomorphism of $M_y$ that restricted to the (one dimensional) degree zero part of $M_y$ is just multiplication by a rational   number $\lambda(l).$ It is easy to see by construction of the tree $\mathbb{T}_x$ (looking at the definition of each of its composing morphisms) that $\lambda(l)\in \mathbb{Z},$ and so we define $F(l)$ as the biggest prime divisor of $\lambda(l).$ In Section \ref{nonex} we prove

\vspace{0.2cm}

\textbf{Theorem 2} \emph{ $\ $If there exists $l  \in \mathcal{A}_x$ such that $F(l)>h$, with $h$ the Coxeter number of $W,$ then Lusztig's conjecture for an algebraic group $G$ with Weyl group $W$ is false in characteristic $p=F(l)$.}

\vspace{0.2cm}

We explain in the next section the idea behind this theorem. This result was strengthened by the fact that, some months after this paper was posted on the arXiv Geordie Williamson \cite{Wi3} found, for each $SL_n,$  leaves $l_n$ with $F(l_n)>n$ when $n>18$. Moreover, $F(l_n)$ seem to grow exponentially in $n$. Williamson proved Theorem 2 independently (see \cite{Wi3}).

 \subsection{How to algorithmically find the bad primes}

After having found the double leaves basis, the next crucial point in this paper is to realize that there is a ``favorite projector" $p_x$ in $\mathrm{End}(M_x)$ projecting to $B_x$ (Section \ref{fav}). It is a simple observation but it took many years to the author to actually do this observation.

Let $W$ be a Weyl group  or an affine Weyl group. If $W$ is a Weyl group we define $W^\circ=W$ and if $W$ is an affine Weyl group, $W^{\circ}$ is the finite subset considered in Section \ref{Sa}.

 For every $x\in {W}^{\circ},$ there is an indecomposable Soergel bimodule $B_x\in \mathcal{B}({W},\mathbb{Q})$ appearing only once in the direct sum decomposition of $M_x$. We find an algorithm to express the projector $p_x$ corresponding to $B_x$  as a linear combination, with coefficients in $\mathbb{Q},$ of the elements in the   double leaves of $\mathrm{End}(M_x)$. Let $d(x)$ be the set of primes dividing the denominators of the above-mentioned coefficients and $$D=\bigcup_{x\in W^\circ}d(x).$$  
%%%%%%%%%%%%%%%%%%%%%%%%%%%%%%%%%%%%
In the category  $ \mathcal{B}(W^\circ,{\mathbb{F}}_p)$ 
(defined as the full subcategory of $\mathcal{B}(W,\mathbb{F}_p)$ with objects direct sums of shifts of objects of the type $B_x$ with $x\in W^\circ$) the double leaves are also well defined (they are defined over $\mathbb{Z}$), so if $p\notin D$ you can reduce mod $p$ the coefficients in the expansion of $p_x$ in terms of the double leaves
so as to produce projectors $p'_x\in  \mathcal{B}(W^\circ,{\mathbb{F}}_p)$ for every $x\in W^\circ$. We prove in Section \ref{0p} that the corresponding bimodules $\mathrm{Im}(p'_x)\in \mathcal{B}(W^\circ,{\mathbb{F}}_p)$ have the same decategorification in $\mathcal{H}$  as $\mathrm{Im}(p_x)\in \mathcal{B}({W}^\circ,\mathbb{Q}).$ As Soergel's $0$-conjecture is true, with this method one can prove that Soergel's $p$-conjecture is true in $ \mathcal{B}(W^\circ,{\mathbb{F}}_p)$ if and only if $p\notin D .$

Let us roughly explain the algorithm mentioned above.   Consider $x,y\in W^\circ$ and  $l, l'$ two degree zero leaves from $x$ to $y$.  Let's say that $M_x=B_{s_1}\cdots B_{s_n}$ and suppose you have constructed $p_{xs_n}$  in the endomorphism ring of $M_{xs_n}=B_{s_1}\cdots B_{s_{n-1}}$. The restriction of the degree zero endomorphism $$l\circ (p_{xs_n}\otimes \mathrm{id})\circ l'^a$$ to the one dimensional degree zero part of $M_y$ is a scalar multiple of the identity that we denote $\lambda(l,l')\in \mathbb{Z}$. For $x, y\in W^\circ$ there exist a subset $L_{x,y}= \{l_1,\ldots,l_n\}$ (essentially the ``linearly independent" elements) of the set  of degree zero leaves from $x$ to $y$ such that if $d(x,y)$ is the determinant of the matrix with $(i,j)$-entry $\lambda(l_i,l_j),$ then we have the following theorem.

\vspace{0.2cm}

\textbf{Theorem 3}  \emph{$\ $  Let $R$ be a root system with affine Weyl group $W$ (resp. Weyl group $W$).   Lusztig's conjecture (resp. Lusztig's conjecture around the Steinberg weight) for algebraic groups with root system $R$ is true if (resp. if and only if)  the characteristic does not divide any element of the set 
$$\{d(x,y)\, |\, x\in W^\circ, \, y\leq x\}$$ where $\leq$ stands for the Bruhat order. 
}
\vspace{0.2cm}
 
Moreover, with all the values $\lambda(l_i,l_j)$ you can construct the projector $p_x,$ so this gives an inductive construction of these projectors.  For details look at Section \ref{ex}.

We remark that (as we said before) we believe that Fiebig's conjecture is equivalent to Lusztig's conjecture and this would imply that in  Theorem 3, when $W$ is an affine Weyl group we could replace ``if" by ``if and only if".  The missing part (the only if) could be relevant only if  we could  find counterexamples to Lusztig's conjecture  using weights that are not around the Steinberg weight.

There existed algorithms to find the bad primes, see \cite{AJS} and \cite{Fi5}.  The algorithm in  \cite{Fi5} is more efficient than the one in \cite{AJS}, in particular it allows to obtain explicit bounds for the bad primes. The algorithm here has the advantage over the one in \cite{Fi5} that we only consider degree zero morphisms, whereas Fiebig, in \cite{Fi5}, has to conisder all degree morphisms. This diminishes considerably calculations. 

In the paper \cite{Wi3} Williamson constructs another algorithm  (quite similar to ours) to find the bad primes. In his algorithm (obtained using intersection forms) he considers $D_{x,y},$ all the degree zero leaves from $x$ to $y$ for  $x,y\in W^{\circ}$ (not just the ones in $L_{x,y}\subseteq D_{x,y}$) and then the bad primes  are the prime numbers for which  the rank of the matrix
$(l'\circ l^a)_{l,l'\in D_{x,y}}$, for all $x,y\in W^{\circ}$ is different that the one in characteristic zero. 

Williamson's algorithm is better in one sense and worst in another sense. It is better in that you don't need to know, for a word of some given length, the projectors to the indecomposables of the words of lesser lengths (in the language above, $p_{xs_n}$). This implies that, if you are only interested in knowing the full set of bad primes for Soergel's conjecture, then his algorithm is better. Our algorithm has the advantage that, although it is more difficult to calculate, it gives more precise information. Using it, one knows the coefficient associated to each double leaf when you expand the favorite projector in terms of double leaves. In general, there might be a relation between the combinatorics of a given double leaf and the aforementioned coefficient that would not be visible in Williamson's algorithm.  See \cite{EL} for an example, the Universal Coxeter system, where the projectors are easier to calculate than the full set of bad primes. 

I hope that this will be the same situation in the affine Weyl group case. So, this new information (the coefficient associated to each double leaf)  might  be useful to find closed formulas for the denominators. One problem to do this is that we should have to be able to find a ``canonical" double leaf (the construction of the tree  depends on some choices and there is still no canonical way to make them). This problem doesn't appear in the Universal Coxeter system case, because in that case the light leaves are indeed canonical. There is some recent work in this direction done by the author and G. Williamson, that solves this difficulty for almost all Weyl groups.

\subsection{Standard Leaves basis}
The recent proof of Kazhdan-Lusztig's positivity conjecture is very important for algebraic combinatorics. Moreover, many results assumed that KL polynomials have positive coefficients. But it remains an important problem: how to prove this in a combinatorial way. In other words, how to interpret these coefficients as the cardinality of some combinatorially defined set.  In section \ref{SLB} we prove a proposition that we believe that might help in this direction. 

Soergel found in \cite{So4} that the morphism $\varepsilon: \mathcal{H}\rightarrow \langle \mathcal{B}\rangle$ 
has the following inverse
\begin{equation}\label{S}
\eta(\langle  M\rangle)=\sum_{y}\underline{\overline{\mathrm{rk}}}\mathrm{Hom}(M,R_y)T_y\ \  \mathrm{for}  \  \mathrm{all}\ y\in W,
\end{equation}
where $\underline{\overline{\mathrm{rk}}}$ is the graded rank of a bimodule and $R_x$ is the standard $(R,R)$-bimodule: as a left module equal to $R$ and with the right action   twisted by $x\in W$ (for details see Section \ref{equiv}).
This formula says that  the KL polynomial $p_{x,y}$ is (modulo normalization) $\underline{\overline{\mathrm{rk}}}\mathrm{Hom}(B_x,R_y)$, and this explains how  Proposition 4 enters this story. 

For every $y\in W$ there is a canonical morphism $\beta_y: M\rightarrow R_y$. If $l$ is a  leaf from $x$ to $y$, we define $\l^{\beta}$ the composition of morphisms $\beta_y\circ l.$ The following is Proposition \ref{Lb}

\vspace{0.2cm}

\textbf{Proposition 4}  $\ $\emph{The set $\l^{\beta}$ with $l$ a leaf from $x$ to $y$ is a basis of the set $\mathrm{Hom}(M, R_y)$ as a right $R$-module. We call it the \emph{Standard leaves basis}.}

\vspace{0.2cm}

\subsection{Structure of the paper}

The structure of the paper is as follows. In Section \ref{Prel} we give the definitions and first properties of Hecke algebras and Soergel bimodules. In Section \ref{Light leaves basis} we construct the double leaves basis and in Section  \ref{ac} we prove that it is indeed a basis for arbitrary Coxeter systems. In Section \ref{ex} we find the above-mentioned algorithm to express $p_x$ as a linear combination of the elements of the double leaves basis and in Section \ref{nonex} we define $L_{x,y}$ and prove Theorem 3. In section \ref{0p} we explain the relation between idempotents in Soergel's theory over $\mathbb{Q}_p$ and idempotents in Soergel's theory over $\mathbb{F}_p$.  In section \ref{SLB} we prove Proposition 4 .  Finally, in Section \ref{eq1} we prove that Soergel's affine $p$-conjecture is equivalent to Fiebig's conjecture in moment graphs theory.

\subsection{Acknowledgements} We would like to thank Geordie Williamson, Fran\c{c}ois Digne and Wolfgang Soergel for helpful comments. 

The results of this paper have been exposed from 2011 to 2013 (Freiburg, Erlangen, Buenos Aires, C\'ordoba, Valpara\'iso, Olmu\'e, Puc\'on, Tucum\'an).

This work was partially supported by the Von Humboldt foundation during the author's postdoctoral stay in Freiburg and by Fondecyt project number 11121118.

\section{Preliminaries}\label{Prel}

In  this section we consider  $(W,\mathcal{S})$  an arbitrary Coxeter system unless explicitly stated. We define the corresponding Hecke algebra and Soergel bimodules.  

\subsection{Hecke algebras}\label{HA}  Let $\mathcal{A}=\mathbb{Z} [v,v^{-1}]$ be the ring of Laurent polynomials with integer coefficients. The \textit{Hecke algebra} $\mathcal{H} = \mathcal{H}(W, \mathcal{S})$
is the $\mathcal{A}$-algebra with generators  $\{
 T_{s}
\}_{s\in \mathcal{S}}$, and relations  $$T^{2}_{s}=
v^{-2}+(v^{-2}-1)T_{s} \mathrm{ \ \ for\ all\ \ } s\in \mathcal{S}\ \   \mathrm{and} $$
$$\underbrace{T_{s}T_{r}T_{s}...}_{m(s,r)\, \mathrm{terms}
}=\underbrace{T_{r}T_{s}T_{r}...}_{m(s,r)\, \mathrm{terms} } \ \mathrm{if}\
s,r \in \mathcal{S}\ \mathrm{ and\ }sr \mathrm{\ is\ of\ order\ } m(s,r).$$

If
$x=s_{1}s_{2}\cdots s_{n}$ is a reduced expression of $x$, we define $T_x=T_{s_1}T_{s_2}\cdots T_{s_n}$ ($T_x$ does not depend on the choice of the reduced expression). The set $\{T_x\}_{x\in W}$ is a basis of the $\mathcal{A}$-module $\mathcal{H}.$ We put $q=v^{-2}$ and $\widetilde{T}_x = v^{l(x)} T_x$.

 There exists a unique ring involution $d: \mathcal{H} \rightarrow \mathcal{H}$ with $d(v) = v^{-1}$ and $ d(T_x ) = (T_{x^{-1}})^{-1}$. Kazhdan and Lusztig  \cite{KL1} proved that for $x \in W$ there exist a unique $C'_x \in \mathcal{H}$ with $d(C'_x) = C'_x$
 and $$C'_x \in \widetilde{T}_x+ \sum_y v\mathbb{Z}[v]\widetilde{T}_y.$$
 The set $\{C'_x\}_{x\in W}$ is the so-called \emph{Kazhdan-Lusztig basis} of the Hecke algebra. It is a basis of $\mathcal{H}$ as an $\mathcal{A}$-module.

\subsection{Reflection representations}\label{Rep}

Recall that $(W,\mathcal{S})$ is an arbitrary Coxeter system. A \emph{reflection faithful representation} of $W$ over a field $k$ as defined by Soergel is a finite dimensional $k$-representation that is faithful, and such that, for $w\in  W$ the fixed point set $V^w$ has codimension one in $V$ if and only if $w$ is a reflection, i.e. a conjugate of a simple relfection (we call this set $\mathcal{T}\subset W$). For $k=\mathbb{R}$ Soergel constructed \cite[Section 2]{So4} such a representation for any Coxeter system. 
%In Sections \ref{Light leaves basis} and \ref{ac} we can use the representation introduced by Soergel. We could however use any other reflection faithful representation as well.

For the purposes of this paper we need to define a class of representations in which all of Soergel's theory works but that are local in nature.
For $x\in W$ consider the reversed graph$$\Gr(x)=\{(xv,v)\,\vert\, v\in V\}\subseteq V\times V.$$

\begin{defi}
Let $W'$ be a subset of $W$ closed under $\leq$ (where $\leq$ is the Bruhat order). This means that  for every $w'\in W'$, we have $$\{ w\in W\,\vert\, w\leq w'\}\subseteq W'.$$
An $n-$dimensional representation $V$ of $W$ is called $W'-$reflection  if the two following conditions hold.
\begin{itemize}
\item For $x,y\in W',$ we have $\mathrm{dim}(\Gr(x)\cap \Gr(y))=n-1$ if and only if  $x^{-1}y\in \mathcal{T}$
\item There is no $x,y,z\in W'$ different elements such that $$\Gr(x)\cap \Gr(y)=\Gr(z)\cap \Gr(y),$$ both sets having dimension $n-1$.
\end{itemize}
\end{defi}

We remark that a $W-$reflection representation $V$ of $W$ is reflection faithful, and that all of Soergel's theory works for a $W$-reflection representation and also for a $W'$-reflection representation in the sense of Proposition \ref{grande} $(5)$.

We now make a brief terminology detour. 

\begin{term}
The term ``geometric representation" defined in \cite[Ch. v, 4.3]{Bo} and used generally in the literature seems flawed to us. This representation is not more geometric than the contragradient one (defined in \cite[Ch. v, 4.4]{Bo}). We propose (with  W. Soergel) to call it \emph{rootic representation} as you can see the lines generated by the roots as
pairwise disjoint $(-1)$-eigenspaces of reflections. We also propose to call \emph{alcovic representation} the contravariant representation, since the alcoves are ``visible".  The group acts faithfully in the set of alcoves, and even simply transitively on those in the Tits cone. 
\end{term}

We proved in \cite{Li2} that all of Soergel's theory over $k=\mathbb{R}$ works as well if you choose the rootic representation (that it is not a reflection faithful representation). 

Let $k$ be a field of characteristic different from $2$ and let $G\supset B\supset T$ be a semisimple split simply connected algebraic group over $k$ with a Borel subgroup and a maximal torus. Let $(W,\mathcal{S})$ be the finite Weyl group of $G\supset B.$ The representation of $W$ on the Lie algebra $\mathrm{Lie}(T)$ is reflection faithful. 

Let $\widehat{W}$ be the affine Weyl group associated to $G\supset B.$ Let $V$ be a realization of the affine Cartan matrix. Fiebig \cite[Section 2.4]{Fi4} considers the following finite subsets of $\widehat{W}$. 
\begin{itemize}
\item $\widehat{W}^{\mathrm{res},+}:=\{w\in \widehat{W}\, \vert\, 0\leq \langle w\cdot 0, \alpha^{\vee}\rangle<p \mathrm{\ for\ all\ simple\ roots\ }\alpha \}$
\item $\widehat{W}^{\mathrm{res},-}:=w_0\widehat{W}^{\mathrm{res},+}$
\item $\widehat{W}^{0}:=\{w\in \widehat{W}\, \vert\,  w\leq \hat{w}_0\}$
\end{itemize}
Where the action $\cdot$ is the $\rho$-shifted action of $\widehat{W}$ on $V$, $w_0$ is the longest element in $ W\subset \widehat{W}$ and finally $\hat{w}_0$ is the longest element in $\widehat{W}^{\mathrm{res},-}$. The set $\widehat{W}^{\mathrm{res},-}$ is called the \emph{antifundamental box.}

We remark that this representation of $\widehat{W}$ is not reflection faithful but $\widehat{W}^\circ-$reflection as defined before .

So, summing up, if we start with the data $G\supset B\supset T$, we have the associated Weyl group $W$ with a representation in $V=\mathrm{Lie}(T).$ If $\widehat{W}$ is an affine Weyl group we take $V$ a realization of the affine Cartan matrix and if $W$ is a general Coxeter group that is not a Weyl nor an affine Weyl group and $k=\mathbb{R}$ we take $V$ as the rootic representation.

\subsection{Soergel bimodules}\label{So}

For any $\mathbb{Z}$-graded object $M=\bigoplus_i M_i,$ and every  $n\in \mathbb{Z}$, we denote by $M(n)$ the shifted object defined by the formula $$(M(n))_i=M_{i+n}.$$

Let $R=R(V)$\label{d1}  be the algebra of regular functions on $V$ with  the following grading: $R=\bigoplus_{i\in
\mathbb{Z}}R_i$ with  $R_2 = V^*$ and $R_i=0$ if $i$ is odd. The 
action of $W$ on $V$ induces an action on $R$. For $s\in\mathcal{S}$ consider the graded $(R,R)-$bimodule $$B_s=R\otimes_{R^s} R(1),$$ where $R^s$ is the subspace of $R$ fixed by $s$.

 The category of  \textit{Soergel bimodules} $\mathcal{B}=\mathcal{B}(W,k)$ is the category  of $\mathbb{Z}-$graded $(R,R)-$bimodules with objects the finite direct sums of direct summands of objects of the type $${B}_{s_1}\otimes_{{R}}{B}_{s_2}\otimes_{{R}}\cdots\otimes_{{R}} {B}_{s_n}(d)$$ for $(s_1,\ldots, s_n)\in \mathcal{S}^n,$ and $d\in\mathbb{Z}.$

Given $M,N\in \mathcal{B}$ we denote  their tensor product simply by juxtaposition: $M N := M \otimes_R N$.

If $ \underline{s}=(s_1,\ldots, s_n)\in \mathcal{S}^n,$ we will denote by $B_{\underline{s}}$ the $({R},{R})-$bimodule $${B}_{s_1}{B}_{s_2}\cdots {B}_{s_n}\cong {R}\otimes_{{R}^{s_1}}{R}\otimes_{{R}^{s_2}}\cdots \otimes_{{R}^{s_n}}{R}(n).$$ We use the convention $B_{\mathrm{id}}=R.$ Bimodules of the type $B_{\underline{s}}$ will be called  \emph{Bott-Samelson bimodules.}

Given a Laurent polynomial with positive coefficients $P = \sum a_iv^i \in \mathbb{N}[v, v^{-1}]$ and a graded bimodule $M$ we set
$$P \cdot M := \bigoplus M(-i)^{\oplus a_i}.$$

For every essentially small additive category $\mathcal{A}$, we call  $\langle\mathcal{A}\rangle$ the \textit{split Grothendieck 
group}. It is the free abelian group generated by the objects of  $\mathcal{A}$ modulo the relations $M=M'+M''$ whenever we have $M\cong M'\oplus M''$. Given an object $A\in \mathcal{A},$ let  $\langle A \rangle$ denote its class in $ \langle \mathcal{A}
\rangle$.

In \cite{So2} Soergel proves that there exists a unique ring isomorphism
$\varepsilon: \mathcal{H}\rightarrow \langle \mathcal{B}\rangle $
such that $\varepsilon(v)=\langle R(1)\rangle$ and $\varepsilon(T_s+1)=\langle R\otimes_{R^s}R \rangle$ for all $ s\in \mathcal{S}. $

\subsection{Support}\label{ca}

For any finite subset $A$ of $W$ consider the union of the corresponding graphs $$\Gr(A)=\bigcup_{x\in A}\Gr(x)\subseteq V\times V.$$ If $A$ is finite, we view  $\Gr(A)$ as a subvariety of $V \times V$. If we identify $R\otimes_{k}R$ with the regular functions on $V\times V$ then $R_A$, the regular functions on $\Gr(A)$, are naturally $\mathbb{Z}$-graded $R$-bimodules. We will also write $R_{\{x\}}=R_{x}.$ One may check that given $x \in W$ the bimodule $R_x$ has the following simple description: $R_x \cong R$ as a left module, and the right action is twisted by $x$: $m \cdot r = mx(r)$ for $m \in R_x$ and $r \in R$.

For any $R$-bimodule $M \in \BMod{R}{R}$ we can view $M$ as an $R \otimes_{k} R$-module (because $R$ is commutative) and hence as a quasi coherent sheaf on $V \times V$. Given any finite subset $A \subseteq W$ we define \[  M_A := \{ m \in M \; | \; \supp m \subseteq \Gr(A) \} \] to be the subbimodule consisting of elements whose support is contained in $\Gr(A)$ (we remark that in \cite{So4} the bimodule $M_A$ is denoted $\Gamma_AM$). For $x\in W$ we denote $ M_{\{x\}}=M_x.$  

Given any Soergel bimodule $M$ we define $M^x$ the restriction of $M$ to $\Gr(x),$ and $M^{x\cap y}$ its restriction to $\Gr(x)\cap \Gr(y).$ For $x,y \in W$ we call $\rho_{x,y}: M^x\rightarrow M^{x\cap y}$ the restriction map.

%We also introduce the notation $M^A=M/M^A.$
In the following we will abuse notation and write $\le x$ for the set $\{ y \in W \; | \; y \le x \}.$

 \subsection{Indecomposable Soergel bimodules}\label{ind}
We start by recalling the most important features of the  indecomposable Soergel bimodules, central in Soergel's theory. Consider $W$ a Weyl group and $\mathcal{B}=\mathcal{B}(W, k)$ with $k=\mathbb{Q}$ or $k=\mathbb{F}_p$. 
%that will be useful in Section \ref{Lus}. 
The following can be found in \cite[Theorem 2]{So2} and  \cite[Satz 6.16]{So4} for infinite fields, or  in \cite{EW3} for any complete local ring (including the case $k=\mathbb{F}_p$).
  
\begin{prop}\label{grande} 
\begin{enumerate} 
\item For all $w\in W$ there is, up to isomorphism, a unique indecomposable bimodule $B_w\in \mathcal{B}$ with support in $\mathrm{Gr}(\le w)$ and such that $B_w^w\cong R_w(l(w)).$
\item The map $(w,i)\mapsto B_w(i)$ defines a bijection from the set
$ W\times \mathbb{Z}$ to the set of indecomposable objects of $\mathcal{B}$ up to isomorphism.
\item\label{3} If $\underline{s}=(s_1,\ldots, s_n)$ is a reduced expression of $w\in W$ then there are polynomials $p_y\in \mathbb{N}[v, v^{-1}]$ such that   $$B_{\underline{s}}\cong B_w\oplus\bigoplus_{y<w}p_y\cdot B_y.$$ 
%\item  If we replace everywhere in this proposition the Weyl group $W$ by  the finite subset $\widehat{W}^{\circ}$  of the affine Weyl group $\widehat{W}$ then point $(1)$ remains true, so if we  replace $\mathcal{B}(W, k)$ by the full subcategory $\mathcal{B}^{\circ}$ of $\mathcal{B}(\widehat{W}, k)$ with objects finite direct sums of elements in the set $$\{B_w(d)\,\vert\, w\in \widehat{W}^{\circ}, d \in \mathbb{Z}\},$$ then all points $(1)$ through $(3)$ remain true.
\end{enumerate}
\end{prop}

Let us state the version of Proposition \ref{grande} for an affine Weyl group $\widehat{W}$. 

\begin{prop}\label{grandes} 
 \begin{enumerate} 
\item For all $w\in \widehat{W}^{\circ}$ there is, up to isomorphism, a unique indecomposable bimodule  $B_w\in \mathcal{B}(\widehat{W},k)$ with support in $\mathrm{Gr}(\le w)$ and such that $B_w^w\cong R_w(l(w)).$ Thus we can define $\mathcal{B}^{\circ}$ as the full subcategory of $\mathcal{B}(\widehat{W},k)$ whose objects are finite direct sums of elements in the set $$\{B_w(d)\,\vert\, w\in \widehat{W}^{\circ}, d \in \mathbb{Z}\}.$$
\item The map $(w,i)\mapsto B_w(i)$ defines a bijection from the set
$ \widehat{W}^{\circ}\times \mathbb{Z}$ to the set of indecomposable objects in $\mathcal{B}^{\circ},$ up to isomorphism.
\item\label{3} If $\underline{s}=(s_1,\ldots, s_n)$ is a reduced expression of $w\in \widehat{W}^{\circ}$ then there are polynomials $p_y\in \mathbb{N}[v, v^{-1}]$ such that   $$B_{\underline{s}}\cong B_w\oplus\bigoplus_{y<w}p_y\cdot B_y.$$ 
\end{enumerate}
\end{prop}

The following proposition is equivalent to Soergel's $0$-conjecture. 
\begin{prop}\label{ind}
Let $k=\mathbb{R}$ and let $W$ be any Coxeter group. For $x,y\in W,$ we have
\begin{equation*}
\underline{\mathrm{Hom}}(B_x,B_y)\cong
\begin{cases}
k& \text{if } x=y,\\
0& \text{otherwise,}
\end{cases}
\end{equation*}
where $\underline{\mathrm{Hom}}(B_x,B_y)$ denote the set of degree zero elements in $\mathrm{Hom}(B_x,B_y).$
\end{prop}

We can now state  \cite[Conjecture 1.13]{So2}, the central conjecture of Soergel.

\begin{conj}[Soergel's $0$-conjecture]\label{cs}
If $k=\mathbb{R}$ and $W$ is any Coxeter system, for every  $w\in W$ we have $\varepsilon(C'_w)=
\langle B_w\rangle$.
\end{conj} 

Its generalization to positive characteristic (not explicitly stated by Soergel but considered in \cite{So3}).

\begin{conj}[Soergel's $p$-conjecture]\label{csp}
If $k=\mathbb{F}_p$ and $W$ is a Weyl group, for every  $w\in W$ we have $\varepsilon(C'_w)=
\langle B_w\rangle$.
\end{conj} 

And now its affine generalization in positive characteristic. This conjecture was never considered by Soergel, but it seems reasonable for us to call it 

\begin{conj}[Soergel's affine $p$-conjecture]\label{cspa}
If $k=\mathbb{F}_p$ and $W$ is an affine Weyl group, for every  $w\in W^{\mathrm{res},-}$ we have $\varepsilon(C'_w)=
\langle B_w\rangle$.
\end{conj}

It is this last conjecture that implies (and should be equivalent to) Lusztig's conjecture.

%%%%%%%%%%%%%%%%%%%%%%%%%%%%%%%%%%%%%%%%%%%%%%%%%%%%%%%%%%%%%%%%%%%%%%%%

\section{Double leaves basis}\label{Light leaves basis}
In this section $(W,\mathcal{S})$ is again an arbitrary Coxeter system.
Let us fix for the rest of this section the sequences $\underline{s}=(s_1,\ldots, s_n) \in \mathcal{S}^n$ and $\underline{r}=(r_1,\ldots, r_p)\in \mathcal{S}^p.$ 
In this section we will define the \emph{double leaves basis} ,  a basis of the space Hom$(B_{\underline{s}}, B_{\underline{r}})$  closely related with the light leaves basis (LLB) constructed in \cite{Li1}. The double leaves basis will be more useful for the purposes of this paper than  the LLB   because of its symmetry properties. We start by recalling the construction of the tree $\mathbb{T}_{\underline{s}}$.

\subsection{The tree $\mathbb{T}_{\underline{s}}$}\label{T}

\subsubsection{Three basic morphisms}  In this subsection we will introduce three important morphisms between Soergel bimodules. 
Let ${x}_s\in V^*$ be an equation of the hyperplane  fixed by $s\in\mathcal{S}$. We have a decomposition ${R}\simeq {R}^s\oplus {x}_s{R}^s,$ corresponding to $${R}\ni p=\frac{p+s\cdot p}{2}+\frac{p-s\cdot p}{2}.$$ We define the \emph{Demazure operator,} a morphism of graded $R^s$-modules
\begin{displaymath}
\begin{array}{lll}\smallskip
\partial_s:R(2)\rightarrow R, && p_1 +x_sp_2 \mapsto p_2.
\end{array}
\end{displaymath}

We will now define three morphisms in ${\mathcal{B}}$ that are the basic ingredients in the construction of the  double leaves basis.
The first one is the multiplication morphism 
\begin{displaymath}\label{d3}
\begin{array}{rll}
{m}_s :{B}_s &\rightarrow &  {R} \\
 {R}\otimes_{{R}^s}{R}(1)\ni p\otimes q &\mapsto & pq \\   
\end{array}
\end{displaymath}
The second one is the only (up to non-zero scalar) degree $-1$ morphism from ${B}_s{B}_s$ to ${B}_s$ :
\begin{displaymath}
\begin{array}{rll}
{j}_s : {B}_s{B}_s &\rightarrow &  {B}_s \\
 \vspace{0.8cm}{R}\otimes_{{R}^s}{R}\otimes_{{R}^s}{R}(2)\ni p\otimes q\otimes r &\mapsto & p{\partial}_s(q)\otimes r \\
\end{array}
\end{displaymath}
Consider the bimodule $X_{sr}=B_sB_rB_s\cdots$ the product having $m(s,r)$ terms (recall that $m(s,r)$ is the order of $sr$).
We define $f_{sr}$ as the only  degree $0$ morphism from $X_{sr}$ to $  X_{rs}$ sending $1\otimes 1\otimes \cdots \otimes 1$ to $1\otimes 1\otimes \cdots \otimes 1$. In \cite{Li3}, \cite{Li4} and \cite{EK} there are different explicit formulas for $f_{sr}$.

\subsubsection{Some choices}\label{Some choices}

%We know that if $x\in W$, $s\in \mathcal{S}$ and $l(xs)<l(x)$, then there exist a reduced expression of $x$ having $s$ as its last element. In \cite[ch.4, \§ 1, prop. 4]{B},

For $x\in W$ consider the set $\mathrm{Rex}(x)$ of all reduced expressions of $x$. The \emph{graph of reduced expressions of}  $x$ (or $\mathrm{gRex}(x)$) is the graph with nodes the elements of $\mathrm{Rex}(x)$ and edges between nodes that are connected by a braid move. It is a known fact that this graph is connected. Moreover, for every couple 
 $(s,\underline{t})$ with  $s\in \mathcal{S}$ and $\underline{t}$ a reduced expression of some $x\in W$ satisfying that $l(xs)<l(x),$ there exists a path in $\mathrm{gRex}(x)$ starting at $\underline{t}$ and ending in an element $\underline{q}$ satisfying  that the last element in the sequence $\underline{q}$ is $s$. In fact there exist many paths satisfying this property, but we choose arbitrarily one of them and  call it $P(s,\underline{t}).$
 
 %$(i,\underline{t})$, with $1\leq i\leq n$, $\underline{t}=(t_1,\ldots ,t_k)$ a reduced expression of $x:=t_1\cdots t_k \in W$ such that $l(x s_i) < l(x)$, the set of sequences of elements of  $\mathrm{Rex}(x)$ $$((t^1_1,\ldots, t^1_k),(t^2_1,\ldots, t^2_k), \ldots ,(t^l_1,\ldots, t^l_k) )$$ where $t^1_i=t_i$ for all $1\leq i \leq k$, $t^l_k=s_i$, and where one passes from $(t^i_1,\ldots, t^i_k)$ to $(t^{i+1}_1,\ldots ,t^{i+1}_k)$ by a braid move, is non-empty. So we choose arbitrarily, for every couple $(n,\underline{t})$ an element of this set, and we call it $P(n,\underline{t})$.

For every couple of  nodes $(\underline{q},\underline{t})$ in $\mathrm{gRex}(x)$ that are connected by an edge (i.e. where $\underline{q}$ and $\underline{t}$ differ by a braid move) there is an associated morphism in $\mathrm{Hom}(B_{\underline{q}}, B_{\underline{t}})$, of the type $\mathrm{id}\otimes f_{sr} \otimes \mathrm{id}$ (we regard the morphism $f_{sr}$ as a categorification of the braid relation). In this way, every directed path in $\mathrm{gRex}(x)$ gives a morphism from the Bott-Samelson bimodule corresponding to the starting node to that corresponding to the ending node. We define  $F_i(\underline{t})$ to be the morphism between Bott-Samelson bimodules associated to $P(s_i,\underline{t}).$

% Let $P(s_i,\underline{t})= (\underline{t}^1,\ldots, \underline{t}^k)$. For every $1\leq l<k$ we have a morphism $f_l$  in $\mathrm{Hom}(B_{\underline{t}^l}, B_{\underline{t}^{l+1}})$, of the type $\mathrm{id}\otimes f_{sr} \otimes \mathrm{id}$  associated to the braid move used to pass from $\underline{t}^{l}$ to $\underline{t}^{l+1}$. We look here the morphism $f_{sr}$ as the categorification of the braid relation. We define $$F_i(\underline{t})=f_{k-1}\circ \cdots \circ f_2 \circ f_1\in \mathrm{Hom}(B_{\underline{t}^1}, B_{\underline{t}^{k}}).$$

 Let us fix, for every $x\in W$ a reduced expression $\underline{s^x}\in \mathrm{Rex}(x)$. Also,  for any reduced expression $\underline{t}\in\mathrm{Rex}(x)$ we fix a directed path in $\mathrm{gRex}(x)$ starting in $\underline{t}$ and ending in $\underline{s^x}.$ We denote by $F(\underline{t},\underline{s^x})$ the corresponding morphism from $B_{\underline{t}}$ to $B_{\underline{s^x}}.$ 

\subsubsection{Construction of  $\mathbb{T}_{\underline{s}}$}\label{Ts}

We construct a perfect binary tree with nodes colored by Bott-Samelson bimodules and arrows colored by morphisms from parent to  child nodes. We construct it by induction on the depth of the nodes. In depth  one we have the following tree:

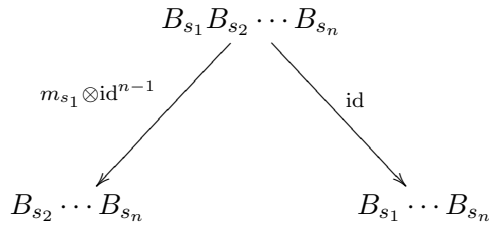
\vspace{0.5cm}

\centerline{
\xymatrix@C=0cm{
&B_{s_1} B_{s_2} \cdots B_{s_n}\ar[ddl]_{m_{s_1}\otimes\mathrm{id}^{n-1}} \ar[ddr]^{ \mathrm{id}} & \\
&& \\
B_{s_2} \cdots B_{s_n}&  & B_{s_1} \cdots B_{s_n}\\
 }
}
\vspace{0.5cm}
Let $k<n$ and $\underline{t}=(t_1,  \cdots, t_{i})\in  \mathcal{S}^{i}$ be such that a node $N$ of depth $k-1$ is colored by the bimodule $(B_{t_1}  \cdots B_{t_{i}})(B_{s_{k}} \cdots B_{s_n}),$ then we have two cases.

\begin{enumerate}
\item If we have the inequality $l(t_1\cdots t_{i}s_k)>l(t_1\cdots t_{i})$, then the child nodes and child edges of $N$ are colored in the following way:
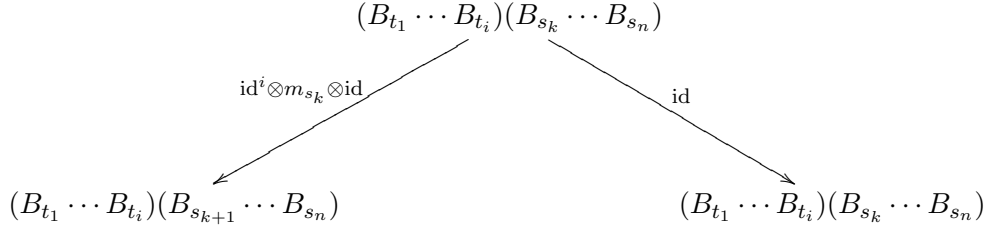
\vspace{0.5cm}

 \centerline{
\xymatrix@C=0cm{
&(B_{t_1}  \cdots B_{t_{i}})(B_{s_{k}} \cdots B_{s_n})\ar[ddl]_{\mathrm{id}^{i}\otimes m_{s_{k}}\otimes\mathrm{id}^{}} \ar[ddr]^{ \mathrm{id}} & \\
&& \\
(B_{t_1}  \cdots B_{t_{i}})(B_{s_{k+1}} \cdots B_{s_n})&  & (B_{t_1}  \cdots B_{t_{i}})(B_{s_{k}} \cdots B_{s_n})\\
 }
}
\item If we have the opposite inequality $l(t_1\cdots t_{i}s_k)<l(t_1\cdots t_{i})$, then the child nodes and child edges of $N$ are colored in the following way (arrows are the composition of the corresponding pointed arrows):
 
 
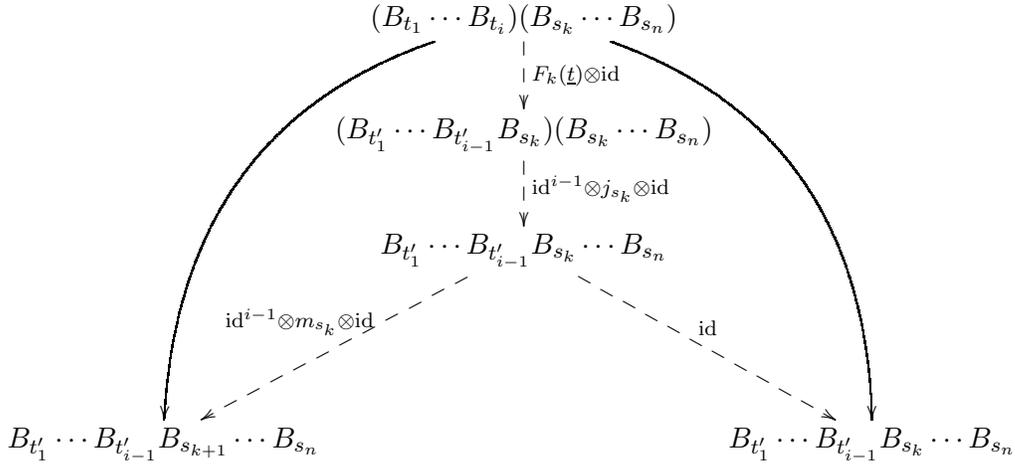
\vspace{0.5cm}
 \centerline{
\xymatrix@C=0cm{ 
& (B_{t_1}  \cdots B_{t_{i}})(B_{s_{k}} \cdots B_{s_n})\ar@/^15mm/[ddddr] \ar@/_15mm/[ddddl]
  \ar@{-->}[d]^{F_k(\underline{t})\otimes \mathrm{id}}&&   \\
&(B_{t'_1}  \cdots B_{t'_{i-1}}B_{s_{k}})(B_{s_{k}} \cdots B_{s_n})\ar@{-->}[d]^{\mathrm{id}^{i-1}\otimes j_{s_{k}}\otimes \mathrm{id}}\\
&    B_{t'_1}  \cdots B_{t'_{i-1}}B_{s_{k}} \cdots B_{s_n}\ar@{-->}[ddl]_{\mathrm{id}^{i-1}\otimes m_{s_{k}}\otimes \mathrm{id}}
 \ar@{-->}[ddr]^{\mathrm{id}}&&  \\
 &&&  \\
B_{t'_1}  \cdots B_{t'_{i-1}}B_{s_{k+1}} \cdots B_{s_n}&& B_{t'_1}  \cdots B_{t'_{i-1}}B_{s_{k}} \cdots B_{s_n}&
 & && &  & &  }
}

\end{enumerate}

In the last step of the construction, i.e. $k=n$, we do exactly the same arrows as in cases $(1)$ and $(2)$ but we end by composing each one of the lower Bott-Samelsons with a morphism of the type $F(\underline{t},\underline{s^x})$ (see Section \ref{Some choices}) with $\underline{t}$ being a reduced expression of $x\in W.$ So we have that each leaf of the tree is colored by a bimodule of the form $B_{\underline{s^x}}$ for some $x\in W.$ This finishes the construction of $\mathbb{T}_{\underline{s}}.$

%For example in case $(1)$, let $\underline{t}=(t_1,  \cdots, t_{p-1})\in  \mathcal{S}^{p-1}$ be such that a node $N$ of depth $n-1$ is colored by the bimodule $B_{t_1}  \cdots B_{t_{p-1}} B_{s_n},$ and let $l(t_1\cdots t_{p-1}s_n)>l(t_1\cdots t_{p-1})$. We set  $\underline{t}'=(t_1,  \cdots, t_{p-1},s_n),$ $ t_1\cdots t_{p-1}=x$  and $t_1\cdots t_{p-1}s_n=y.$
%\vspace{0.5cm}
 %\centerline{
%\xymatrix@C=0cm{
%&B_{t_1}  \cdots B_{t_{p-1}}B_{s_n}\ar@/^5mm/[ddl] \ar@/_5mm/[ddr] \ar@{-->}[dl]_{\mathrm{id}^{p-1}\otimes m_{s_{n}}} \ar@{-->}[dr]^{ \mathrm{id}} & \\
%B_{t_1}  \cdots B_{t_{p-1}}\ar@{-->}[d]_{F(\underline{t},\underline{x})}&  & B_{t_1}  \cdots B_{t_{p-1}}B_{s_n}\ar@{-->}[d]^{F(\underline{t'},\underline{y})}\\
%B_{\underline{x}}  &  & B_{\underline{y}}  \\
 %}
%}

By composition of the corresponding arrows we can see every leaf  of the tree $\mathbb{T}_{\underline{s}}$ colored by $B_{\underline{s^x}}$ as a morphism in the space $\mathrm{Hom}(B_{\underline{s}},B_{\underline{s^x}}).$
Consider the set $\mathbb{L}_{\underline{s}}(\mathrm{id}),$ the leaves of  $\mathbb{T}_{\underline{s}}$ that are colored by the bimodule $R$.  
 In \cite{Li1}  the set $\mathbb{L}_{\underline{s}}(\mathrm{id})$ is called \textit{light leaves basis} and the following theorem is proved.

\begin{thm}\label{LLB}
The set $\mathbb{L}_{\underline{s}}(\mathrm{id})$ is a basis of $\mathrm{Hom}(B_{\underline{s}},R)$ as a left $R$-module.
\end{thm}

\subsection{Construction of the double leaves basis}\label{LL}

In \cite{Li1} the problem of finding a basis of the space  $\label{h}H=\mathrm{Hom}(B_{\underline{s}},B_{\underline{r}})$ was solved using Theorem \ref{LLB} and using repeatedly the adjunction isomorphism \begin{equation}\label{a}\mathrm{Hom}(MB_s,N)\cong \mathrm{Hom}(M, NB_s),\end{equation} which is an isomorphism of  graded left $R$-modules. This gives one basis that we called  \emph{light leaves basis} in \cite{Li1}  of the space $H.$  In this section we will introduce a new basis, that we consider to be the natural generalization of the construction made in Section \ref{T} and that we will call the double leaves basis of $H.$ This basis can be thought of as follows. Take $\mathbb{T}_{\underline{s}}$  and ``paste" it with the  the  tree  $\mathbb{T}_{\underline{r}}$ inverted. This is the image we have to keep in mind. 

Let us be more precise. We need to introduce the adjoint morphisms (in the sense of the adjunction (\ref{a})) of $m_s, j_s,  f_{sr}$. The adjoint of $m_s$ is the morphism
\begin{displaymath}
\begin{array}{rll}
  {\epsilon}_s :{R}&\rightarrow &  {B}_s \\
1 &\mapsto & {x}_s\otimes 1+1\otimes  {x}_s.
\end{array} 
\end{displaymath}

The adjoint of $j_s$ is the morphism 

\begin{displaymath}
\begin{array}{rll}
  p_s :B_s &\rightarrow &  {B}_s{B}_s \\
a\otimes b &\mapsto & a\otimes 1\otimes b.
\end{array}
\end{displaymath}

And finally the adjoint morphism of $f_{sr}$ is $f_{rs}$.

 For any leaf $f:B_{\underline{r}}\rightarrow B_{\underline{s^x}}$ in $\mathbb{T}_{\underline{r}}$ we can find its \emph{adjoint leaf} $f^a:B_{\underline{s^x}} \rightarrow  B_{\underline{r}}$ by replacing each morphism in the set $\{m_s, j_s, f_{sr}\}$ by its adjoint. So we obtain a tree $\mathbb{T}_{\underline{r}}^a$ where the arrows go from children to parents.

Let $x,y\in W.$ If $f\in\mathrm{Hom}(B_{\underline{s}},B_{\underline{s^x}})$ and $g \in\mathrm{Hom}(B_{\underline{s^y}},B_{\underline{r}})$, we define
$$g\cdot f=
\begin{cases}
\hspace*{0cm} g\circ f \hspace*{0.1cm} \mathrm{if}\ x=y\\
\hspace*{0.3cm} \emptyset\hspace*{0.4cm} \mathrm{if}\ x\neq y
\end{cases}$$

Let  $\mathbb{L}_{\underline{s}}$ be the set of leaves of $\mathbb{T}_{\underline{s}},$ this is

$$\mathbb{L}_{\underline{s}}\subset \coprod_{x\in W}\mathrm{Hom}(B_{\underline{s}},B_{\underline{x}})$$

 We call the set $\mathbb{L}_{\underline{r}}^a\cdot \mathbb{L}_{\underline{s}}$ the \emph{double leaves basis}  of $\mathrm{Hom}(B_{\underline{s}}, B_{\underline{r}}).$  
 \begin{thm}\label{LL}
The double leaves basis  is a basis as  left $R$-module of the space $\mathrm{Hom}(B_{\underline{s}}, B_{\underline{r}}).$ 
\end{thm}
 
 The proof of Theorem \ref{LL} is very similar to the proof of \cite[Th\'eor\`eme 5.1]{Li1}. We will write down all the details  in Section \ref{ac}.

%%%%%%%%%%%%%%%%%%%%%%%%%%%%%%%%%%%%%%%%%%%%%%%%%%%%%%%%%%%%%%%%%%%%%%%%
%%%%%%%%%%%%%%%%%%%%%%%%%%%%%%%%%%%%%%%%%%%%%%%%%%%%%%%%%%%%%%%%%%%%%%%%
%%%%%%%%%%%%%%%%%%%%%%%%%%%%%%%%%%%%%%%%%%%%%%%%%%%%%%%%%%%%%%%%%%%%%%%%
%%%%%%%%%%%%%%%%%%%%%%%%%%%%%%%%%%%%%%%%%%%%%%%%%%%%%%%%%%%%%%%%%%%%%%%%
%%%%%%%%%%%%%%%%%%%%%%%%%%%%%%%%%%%%%%%%%%%%%%%%%%%%%%%%%%%%%%%%%%%%%%%%
%%%%%%%%%%%%%%%%%%%%%%%%%%%%%%%%%%%%%%%%%%%%%%%%%%%%%%%%%%%%%%%%%%%%%%%%
%%%%%%%%%%%%%%%%%%%%%%%%%%%%%%%%%%%%%%%%%%%%%%%%%%%%%%%%%%%%%%%%%%%%%%%%
%%%%%%%%%%%%%%%%%%%%%%%%%%%%%%%%%%%%%%%%%%%%%%%%%%%%%%%%%%%%%%%%%%%%%%%%
%%%%%%%%%%%%%%%%%%%%%%%%%%%%%%%%%%%%%%%%%%%%%%%%%%%%%%%%%%%%%%%%%%%%%%%%

\section{Soergel's conjecture for affine Weyl groups}\label{Lus}
In  sections \ref{fav} through \ref{nonex}  we consider ${W}$  an arbitrary Coxeter system and  $k$ a field of characteristic zero. 
% and we fix a sequence $\underline{s}=(s_1,\ldots, s_n) \in \mathcal{S}^n$. 
 We will find a recursive algorithm to find the set of ``good primes" for Soergel's conjecture.  

 \subsection{The favorite projector in any Bott-Samelson}\label{fav} 
 Let $x\in {W}$and $s\in \mathcal{S}$ be such that $l(xs)>l(x).$ Then we have the following formula in the Hecke algebra
 $$C'_xC'_s=C'_{xs}+\sum_{y<xs}m_yC'_y,\hspace{0.5cm}\mathrm{with}\  m_y\in \mathbb{N}.$$ 
 As Soergel's $0$-conjecture is now a theorem, this formula implies
 \begin{equation}B_xB_s\cong B_{xs}\oplus \bigoplus_{y<xs}B_y^{\oplus m_y},\hspace{0.5cm}\mathrm{with}\  m_y\in \mathbb{N}.\end{equation}
 By Proposition \ref{ind} (recall that we are working with a field of characteristic $0$) we see that there is only one projector $p_{x,s}^{xs}$ in $\mathrm{End}(B_xB_s)$ whose image is $B_{xs}.$ Moreover, if $z\in Z:=\{y\in W\,\vert\, m_y\neq 0\},$ there is only one projector $p_{x,s}^z$ in $\mathrm{End}(B_xB_s)$ whose image is $B_{z}^{\oplus m_z}.$
 
By the previous paragraph, for $\underline{t}$  a reduced expression of $y\in W,$ there is an obvious way to define by induction on the length of $\underline{t}$ a projector $p_{\underline{t}}$ in $\mathrm{End}(B_{\underline{t}})$ whose image is isomorphic to $B_{y}$. Let us give the details. If $\underline{t}=(t_1)$ then $p_{(t_1)}=\mathrm{id}\in\mathrm{End}(B_{t_1})$. 

Let $\underline{t}=(t_1,\ldots, t_k) \in \mathcal{S}^k$ be a reduced expression of $y\in W$ and suppose that we have defined $p_{\underline{t}'}$ in $B_{t_1}\cdots B_{t_{k-1}},$ with $\underline{t}'=(t_1,\ldots, t_{k-1})$ and $x=t_1\ldots t_{k-1}.$  As the image of $p_{\underline{t}'}\otimes \mathrm{id}\in \mathrm{End}(B_{\underline{t}})$ is isomorphic to $B_xB_{t_k}$, we define 
 $$p_{\underline{t}}=p_{x,t_k}^{y}\circ (p_{\underline{t}'}\otimes \mathrm{id}).$$
 %where $p_{x,s_n}$ is in principle defined over the image of $p_{\underline{s}'}\otimes \mathrm{id}$, but we extend it to all $B_{\underline{s}}$ by imposing it to be zero in the kernel of $p_{\underline{s}'}\otimes \mathrm{id}.$

 \subsection{Explicit construction of the favorite projector}\label{ex}
 We will construct the favorite projector by induction in the length of the Bott-Samelson. We want to find the favorite projector of the expression $\underline{s}=(s_1,\ldots, s_n)$ and we suppose that we have explicitly constructed the favorite projector of every Bott-Samelson having less than $n$ terms. Let $\underline{s}'=(s_1,\ldots, s_{n-1}),$  $x=s_1\cdots s_{n-1}\in {W}$ and $s_n=s.$ 
  
If $\underline{q}$ and $\underline{t}$ are reduced expressions of some elements in $W$ and $f \in \mathrm{Hom}(B_{\underline{q}},B_{\underline{t}})$ then we denote $$pf:=p_{\underline{t}}\circ f \in \mathrm{Hom}(B_{\underline{q}},B_{\underline{t}}),$$ and $$fp:=f\circ p_{\underline{q}} \in \mathrm{Hom}(B_{\underline{q}},B_{\underline{t}}).$$
 We will see in Section \ref{ac} (Corollary \ref{p})  that the set $\mathbb{L}_{\underline{r}}^a\cdot p\mathbb{L}_{\underline{s}}$  is also a basis of  $\mathrm{Hom}(B_{\underline{s}}, B_{\underline{r}}).$ We  call it \emph{$p$-double leaves basis}.

  Let $P=p_{\underline{s}'}\otimes \mathrm{id}\in \mathrm{End}(B_{\underline{s}})$. By definition of $p_{\underline{s}'}$ we have  $\mathrm{Im}(P)\cong B_xB_s.$ Let us define the numbers $m_y\in \mathbb{N}$ by the isomorphism
$$\mathrm{Im}(P)\cong B_{xs}\oplus \bigoplus_{y<xs}B_y^{\oplus m_y}.$$
We introduce some notation.
\begin{nota}
If $l,l'\in \mathbb{L}_{\underline{s}}$ and their targets agree,  we denote by $(l;l')$ the element $\left.P \, l^a\cdot pl'\right|_{\mathrm{Im}(P)}\in \mathrm{End}(\mathrm{Im}P).$ If the targets of $l$ and $l'$ do not agree, we define  $(l;l')$ as the empty set.
\end{nota}

\begin{nota}
 Let $X$ be a set of homogeneous elements of a graded vector space. We define $X_0$ as the subset of degree zero elements of $X$.
\end{nota}

It is clear that the set $(\mathbb{L}_{\underline{s}};\mathbb{L}_{\underline{s}})_0$ 
 generates  $\underline{\mathrm{End}}(B_{x}B_{s})\subseteq \underline{\mathrm{End}}(B_{\underline{s}}).$  

%By definition we have $\mathbb{L}_{\underline{s}}^a\cdot p\mathbb{L}_{\underline{s}}=\mathbb{L}_{\underline{s}}^ap\cdot \mathbb{L}_{\underline{s}},$

Let $l'\cdot p l\in (\mathbb{L}_{\underline{s}}^a\cdot p\mathbb{L}_{\underline{s}})_0.$ If $\mathrm{deg}(l)<0$ then by Proposition \ref{ind} , we have $\left. pl\right|_{\mathrm{Im}(P)}=0.$ On the other hand, if $\mathrm{deg}(l')<0$ then by Proposition \ref{ind}, we have $ Pl'p=0.$
So we deduce the equality $$(\mathbb{L}_{\underline{s}};\mathbb{L}_{\underline{s}})_0=\left.P \,(\mathbb{L}_{\underline{s}}^a)_0\cdot p(\mathbb{L}_{\underline{s}})_0\right|_{\mathrm{Im}(P)}$$

\begin{defi}For $x\in W$ we define $\mathbb{L}_{\underline{s}}(x)$ as the subset of $\mathbb{L}_{\underline{s}}$ consisting of the elements belonging to $\mathrm{Hom}(B_{\underline{s}},B_{\underline{x}})$.
For $z\in Z:=\{y\in W\,\vert\, m_y\neq 0\}$ define  $\mathbb{L}_{z}$ as the set of leaves $l\in (\mathbb{L}_{\underline{s}}(z))_0$ satisfying that $\left.pl\right|_{\mathrm{Im}(P)}\neq 0.$\end{defi}
\begin{lem}\label{P}
The set $(\mathbb{L}_{{z}};\mathbb{L}_{{z}})$
 generates  $\underline{\mathrm{End}}(B_{z}^{\oplus m_z})\subseteq \underline{\mathrm{End}}(B_{\underline{s}}),$ where we see the bimodule $B_{z}^{\oplus m_z}$ included in $B_{\underline{s}}$ via the formula $$B_{z}^{\oplus m_z}=\mathrm{Im}(p_{x,s}^{z}\circ (p_{\underline{s}'}\otimes \mathrm{id})).$$  
\end{lem}
\begin{proof} 
If $l\in (\mathbb{L}_{\underline{s}}(z))_0$, by Proposition \ref{grande} (3) we know that both the images of  $\left.pl\right|_{\mathrm{Im}(P)}$ and of $Pl^ap_{\underline{s^z}}$ are either  isomorphic to $B_z$ or to zero. On the other hand, by adjunction arguments $l\in \mathbb{L}_{z}$ if and only if $l\in (\mathbb{L}_{\underline{s}}(z))_0$ and  $Pl^ap_{\underline{s^z}}\neq 0,$
 so we conclude that  $l_1,l_2\in (\mathbb{L}_{\underline{s}}(z))_0$ satisfy that $(l;l')\neq 0,$ if and only if   $l, l' \in\mathbb{L}_{z}.$
\end{proof}

A consequence of Lemma \ref{P} is that $$\sum_{l\in \mathbb{L}_{z}}\mathrm{Im}(P \,l^a p_{\underline{s^z}})=B_{z}^{\oplus m_z}\subseteq B_{\underline{s}}.$$

Let $ \mathbb{L}_{z}^{\oplus}=\{l_1,\ldots, l_b \}\subseteq \mathbb{L}_{z}$ be such that $$\bigoplus_{l\in \mathbb{L}_{z}^{\oplus}}\mathrm{Im}(P \,l^a p_{\underline{s^z}})=B_{z}^{\oplus m_z}\subseteq B_{\underline{s}}.$$
This set can be explicitly constructed. Let us be more precise.  We have that $\underline{\mathrm{End}}(B_z^{\oplus m_z})$ can be identified with the full ring of $m_z$ by $m_z$ matrices with coefficients in $k$. If we have that $l_1,\ldots, l_p$ belong to $ \mathbb{L}_{z}^{\oplus}$ we need to know if a new element $l$ of $ \mathbb{L}_{z}$ will belong or not to $ \mathbb{L}_{z}^{\oplus}$. It is easy to see that $\mathrm{Im}(P \,l^a p_{\underline{s^z}})$ is included or intersects only at zero the set $\bigoplus_{l=1}^{p}\mathrm{Im}(P \,l_i^a p_{\underline{s^z}}).$ To distinguish between this cases we just need to evaluate at one element, for example the minimum degree element in $B_{\underline{s^z}}$ (this element goes to itself under $p_{\underline{s^z}}$).

For $l\in \mathbb{L}_{z}$ we have seen that $\mathrm{Im}(P \,l^ap_{\underline{s^z}})\cong B_{z},$ so for $1\leq i\leq b$ we define $$\mathrm{Im}(P \,l_i^a p_{\underline{s^z}})=B_{z}^i\subseteq B_{\underline{s}}.$$
We will find the projector $p_z^i$ in $\mathrm{End}(B_{\underline{s}})$ onto $B_{z}^i$ as a linear combination of elements in the set  $(l_i;\mathbb{L}_{z}^{\oplus}).$
Define $$p_z^i=\sum_{j=1}^{b}\eta^{ij}_z(l_i;l_j).$$
We have the equation $$(l_{i'};l_j)\circ (l_i;l_{j'})=\lambda^{ij}_z (l_{i'};l_{j'}),$$
where $\lambda^{ij}_z\in \mathbb{Z}$ are uniquely defined by the equations 
$$ \lambda^{ij}_zp_{\underline{s^z}}=pl_j P \,l_{i}^ap_{\underline{s^z}}.$$
The fact that the sum $\oplus_i B_z^i$ is direct translates into the set of equations $$p_z^i p_z^j=\delta_{ij}p_z^i,$$ and this translates into the matrix equation \begin{equation}\label{lamda}\lambda_z \eta_z=\mathrm{Id},\end{equation} where $\lambda_z$ is the matrix with $(i,j)$ entry $\lambda^{ij}_z$ and 
$\eta_z$ is the matrix with $(i,j)$ entry $\eta^{ij}_z.$ As $p_z=\sum_{i=1}^b p_z^i$ is the projector onto $B_{z}^{\oplus m_z},$ we finally obtain the projector we ware searching for
 $$p_{\underline{s}}=P-\sum_{z\in Z}p_z.$$

 \subsection{Non-explicit construction of the favorite projector}\label{nonex}
In this section we give a non explicit construction of the favorite projector. We think that this construction if better suited for some calculations, for example, for finding bounds on the bad primes, and could eventually be made explicit by using perverse filtrations in Soergel bimodules.  

 We use the same notations and hypothesis as in Section \ref{ex}. For $ y\in W$ let $t_y\in \mathbb{N}[v,v^{-1}]$ be such that we have an isomorphism $$B_{s_1}\cdots B_{s_n}\cong \bigoplus_{y\in W}t_y B_y.$$
Of course, this isomorphism is not unique but we choose one. Let ${P}'$ be the projector onto the bimodule $$\bigoplus_{y\in W}t_y(0) B_y=B_{xs}\oplus\bigoplus_{z\in Z'}B_z^{\oplus m_z'}$$
 with $Z'=\{y\in W \, \vert \, t_y(0)\neq 0\}$. We have that ${m}_z'\geq m_z$ if $z\in Z\subseteq Z'$. 

Now repeat all the reasoning of section \ref{ex} but replacing $P$ by ${P}'$, $Z$ by $Z'$ and $m_z$ by ${m}_z'$. We have the corresponding objects ${\mathbb{L}}_z', ({\mathbb{L}}^{\oplus}_z)'$  and $ \lambda'_z.$
We obtain the formula $$p_{\underline{s}}={P}'-\sum_{z\in Z'}{p}_z.$$ 
The projector ${P'}$ is not constructed inductively, so the reader might ask why to do a similar, but non-explicit  construction? The important point is how to calculate the coefficients of ${\lambda}'_z.$ If $ ({\mathbb{L}}_{z}^{\oplus})'=\{{l}'_1,\ldots, {l}'_{{c}} \}\subseteq {\mathbb{L}}_{z}',$ for $1\leq i, j \leq {c}$ we have that ${(\lambda'_z)}^{ij}\in \mathbb{Z}$ is, as before, defined by the equation 
$$ {(\lambda'_z)}^{ij}{p}_{\underline{s^z}}=pl'_j {P'} \,(l'_{i})^ap_{\underline{s^z}}.$$
  As $\ker({P'})$ has no direct summand isomorphic to $B_{xs}$ nor to $B_z$ for any $z\in Z',$ by
Proposition \ref{grande} (2) we know that $$p{l}_j \mathrm{ker}({P}) \,{l}_{i}^ap_{\underline{s^z}}=0,$$ so we obtain the  simpler formula $$ {(\lambda'_z)}^{ij}{p}_{\underline{s^z}}=pl'_j  \,(l'_{i})^ap_{\underline{s^z}}.$$

So, in particular when $\mathbb{L}_{z}'$ consists of only one element (so $m'_z=1$), we obtain Theorem 2 in the introduction just by noticing that the lowest degree element of a Bott-Samelson (the element $1\otimes 1\otimes \cdots \otimes 1$) belongs to the biggest indecomposable, by degree considerations.

\subsection{From characteristic zero to  positive characteristic}\label{0p}

In this section we consider $W$ to be a Weyl group (or an affine Weyl group, in which case you have to replace everywhere $W$ by $W^\circ$). We have defined a representation $V=\oplus_{s\in \mathcal{S}}\mathbb{Z}e_s$ of $W$ over $\mathbb{Z}$. We define the symmetric algebra $R$ of $V^*=\mathrm{Hom}(V, \mathbb{Z}).$

We can define the category $\mathcal{B}_{\mathbb{Z}}$ of Soergel bimodules over $\mathbb{Z}$ just as before, with objects directs sums of direct summands of shifts of elements of the form $R\otimes_{R^s}R\otimes \otimes_{R^r}R\otimes\cdots \otimes_{R^t}R$ for $sr\cdots t$ a reduced expression in $W$. Let $\mathbb{F}_p$ be the finite field with $p$ elements, $\mathbb{Z}_p$ the $p$-adic integers and $\mathbb{Q}_p$ the $p$-adic numbers. The categories $\mathcal{B}_{\mathbb{F}_p}$ (resp. $\mathcal{B}_{\mathbb{Z}_p}$, $\mathcal{B}_{\mathbb{Q}_p}$) have objects $b_{\mathbb{Z}}\otimes_{\mathbb{Z}}\mathbb{F}_p$
(resp. $b_{\mathbb{Z}}\otimes_{\mathbb{Z}}\mathbb{Z}_p$, $b_{\mathbb{Z}}\otimes_{\mathbb{Z}}\mathbb{Q}_p$) with $b_{\mathbb{Z}}\in \mathcal{B}_{\mathbb{Z}}.$

As the double leaves is an $\mathbb{F}_p$ and $\mathbb{Q}_p$ basis respectively of the Hom spaces between the Bott-Samelsons in $\mathcal{B}_{\mathbb{F}_p}$ and $\mathcal{B}_{\mathbb{Q}_p}$, by Nakayama's lemma it is a ${\mathbb{Z}_p}-$basis of the Hom's in $\mathcal{B}_{\mathbb{Z}_p}.$

Consider the following functors 
\begin{itemize}
\item $\otimes_{\mathbb{Z}_p}\mathbb{F}_p: \mathcal{B}_{\mathbb{Z}_p}\rightarrow \mathcal{B}_{\mathbb{F}_p}$
\item $\otimes_{\mathbb{Z}_p}\mathbb{Q}_p: \mathcal{B}_{\mathbb{Z}_p} \rightarrow \mathcal{B}_{\mathbb{Q}_p}.$
\end{itemize}

Let $B_{\underline{s}}$ be a Bott-Samelson in $\mathcal{B}_{\mathbb{Z}}$, with $\underline{s}$ a reduced expression of $x\in W.$ Let $p_{\underline{s}}\in \mathrm{End}_{\mathcal{B}_{\mathbb{Q}_p}}(B_{\underline{s}}\otimes_{\mathbb{Z}_p}\mathbb{Q}_p)$ be the favorite projector onto $B_x$ as defined in Section \ref{fav}. There are two cases.

\begin{enumerate}
\item Let us suppose that in the expansion of $p_{\underline{s}}$ in the light leaves the coefficients have no denominators that are divisible by $p$. Then we can lift $p_{\underline{s}}$ to another primitive idempotent in $\mathrm{End}_{\mathcal{B}_{\mathbb{Z}_p}}(B_{\underline{s}}\otimes_{\mathbb{Z}}\mathbb{Z}_p)$ and then apply the functor $\otimes_{\mathbb{Z}_p}\mathbb{F}_p$ to obtain a primitive idempotent $p'_{\underline{s}}$ in $\mathrm{End}_{\mathcal{B}_{\mathbb{F}_p}}(B_{\underline{s}}\otimes_{\mathbb{Z}_p}\mathbb{F}_p)$.

For what follows the reader should read the notations in Section \ref{Some notations}. Let $M,N\in \mathcal{B}_{\mathbb{Z}}.$ As 
$$ \underline{\mathrm{dim}} \mathrm{Hom}_{\mathcal{B}_{\mathbb{F}_p}}(M,N)= \underline{\mathrm{rk}} \mathrm{Hom}_{\mathcal{B}_{\mathbb{Z}_p}}(M,N)= \underline{\mathrm{dim}} \mathrm{Hom}_{\mathcal{B}_{\mathbb{Q}_p}}(M,N),$$ 
(where we consider $M$ and $N$ inside the brackets as tensored with the corresponding field), we have that $$\eta_{\mathbb{Q}_p}(\mathrm{Im}(p_{\underline{s}}))=\eta_{\mathbb{F}_p}(\mathrm{Im}(p'_{\underline{s}}))\in \mathcal{H}.$$

\item Let us suppose that in the expansion of $p_{\underline{s}}$ in the light leaves basis there is at least one coefficient that has a denominator that is divisible by $p$. We will see that very idempotent in $\mathrm{End}_{\mathcal{B}_{\mathbb{F}_p}}(B_{\underline{s}}\otimes_{\mathbb{Z}_p}\mathbb{F}_p)$ can be lifted to an idempotent in $\mathrm{End}_{\mathcal{B}_{\mathbb{Z}_p}}$. In \cite[Th. 3.1 (b), (d)]{Th} this is proven in the case of an algebraically closed residue field. To see this for $\mathbb{F}_p$ we just have to remark that any finitely generated $\mathbb{Z}_p-$module $M$ is topologically complete and that the degree zero part of $\mathrm{End}_{\mathcal{B}_{\mathbb{Z}_p}}(B_{\underline{s}})$ is finite dimensional. 

Then, every idempotent in $\mathrm{End}_{\mathcal{B}_{\mathbb{F}_p}}(B_{\underline{s}}\otimes_{\mathbb{Z}_p}\mathbb{F}_p)$ can be lifted to an idempotent in $\mathrm{End}_{\mathcal{B}_{\mathbb{Q}_p}}(B_{\underline{s}}\otimes_{\mathbb{Z}_p}\mathbb{Q}_p)$, passing through $\mathbb{Z}_p$. Then Soergel's conjecture over $\mathbb{F}_p$ can not be true for all $y\leq x$ because if this was true then $B_{\underline{s}}\otimes_{\mathbb{Z}_p}\mathbb{Q}_p$ and $B_{\underline{s}}\otimes_{\mathbb{Z}_p}\mathbb{F}_p$ would have the same number of indecomposable summands but this can not be true because $\mathrm{End}_{\mathcal{B}_{\mathbb{Q}_p}}(B_{\underline{s}}\otimes_{\mathbb{Z}_p}\mathbb{Q}_p)$ has all the lifted idempotents of $\mathrm{End}_{\mathcal{B}_{\mathbb{F}_p}}(B_{\underline{s}}\otimes_{\mathbb{Z}_p}\mathbb{F}_p)$ and also $p_{\underline{s}}$. The fact that $p_{\underline{s}}$ is different from the other idempotents mentioned is due to the fact that the other idempotents as are lifted can not have a coefficient divisible by $p$ in the expansion of them in the light leaves basis. 
\end{enumerate}

So we conclude  that Soergel's $p-$conjecture (resp. Soergel's affine $p-$conjecture) is true if and only if $p$ does not belong to the set of primes that divide at least one of the denominators of the coefficients appearing in the expansion of all the $p_{\underline{s}}$ for $\underline{s}$ a reduced expression of an element of $W$ (resp. $W^\circ$).

%%%%%%%%%%%%%%%%%%%%%%%%%%%%%%%%%%%%%%%%%%%%%%%%%%%%%%%%%%%%%%%%%%%%%%%%%%%%%%%%%%%%%%%%%%%%%%%%%%%%%%%%%%%%%%%%%%%%%%%%%%%%%%%%%%%%%%%%%%%%%%%%%%%%%%%%%%%%%%%%%%%%%%%%%%%%%%%%%%%%%%%%%%%%%%%%%%%%%%%%%%%%%%%%%%%%%%%%%%%%%%%%%%%%%%%%%%%%%%%%%%%%%%%%%%%%%%%%%%%%%%%%%%%%%%%%%%%%%%%%%%%%%%%%%%%%%%%%%%%%%%%%%%%%%%%%%%%%%%%%%%%%%%%%%%%%%%%%%%%%%%%%%%%%%%%%%%%%%%%%%%%%%%%%%%%%%%%%%%%%%%%%%%%%%%%%%%%%%%%%%%%%%%%%%%%%%%%%%%%%%%%%%%%%%%%%%%%%%%%%%%%%%%%%%%%%%%%%%%%%%%%%%%%%%%%%%%%%%%%%%%%%%%%%%%%%%%%%%%%%%%%%%%%%%%%%%%%%%%%%%%%%%%%%%%%%%%%%%%%%%%%%%%%%%%%%%%%%%%%%%%%%%%%%%%%%%%%%%%%%%%%%%%%%%%%%%%%%%%%%%%%%%%%%%%%%%%%%%%%%%%%%%%%%%%%%%%%%%%%%%%%%%%%%%%%%%%%%%%%%%%%%%%%%%%%%%%%%%%%%%%%%%%%%%%%%%%%%%%%%%%%%%%%%%%%%%%%%%%%%%%%%%%%%%%%%%%%%%%%%%%%%%%%%%%%%%%%%%%%%%%%%%%%%%%%%%%%%%%%%%%%%%%%%%%%%%%%%%%%%%%%%%%%%%%%%%%%%%%%%%%%%%%%%%%%%%%%%%%%%%%%%%%%%%%%%%%%%%%%%%%%%%%%%%%%%%%%%%%%%%%%%%%%%%%%%%%%%%%%%%%%%%%%%%%%%%%%%%%%%%%%%%%%%%%%%%%%%%%%%%%%%%%%%%%%%%%%%%%%%%%%%%%%%%%%%%%%%%%%%%%%%%%%%%%%%%%%%%%%%%%%%%%%%%%%%%%%%%%%%%%%%%%%%%%%%%%%%%%%%%%%%%%%%%%%%%%%%%%%%%%%%%%%%%%%%%%%%%%%%%%%%%%%%%%%%%%%%%%%%%%%%%%%%%%%%%%%%%%%%%%%%%%%%%%%%%%%%%%%%%%%%%%%%%%%%%%%%%%%%%%%%%%%%%%%%%%%%%%%%%%%%%%%%%%%%%%%%%%%%%%%%%%%%%%%%%%%%%%%%%%%%

\section{The double leaves basis is a basis}\label{ac} 
In this section we prove Theorem \ref{LL}.
 We  fix for the rest of this section the Coxeter system $(W,\mathcal{S})$ and the sequences $\underline{s}=(s_1,\ldots, s_n) \in \mathcal{S}^n$ and $\underline{r}=(r_1,\ldots, r_p)\in \mathcal{S}^p.$ 
We first prove that the graded degrees are the correct ones.

\subsection{Some notation}\label{Some notations}
The notations introduced in this section will be useful to understand what we mean when we say that the graded degrees are the correct ones.

  Given a 
$\mathbb{Z}$-graded vector space  $V=\bigoplus_i V_i$, with dim$(V)<\infty$, we define its graded dimension by the formula
$$ \underline{\mathrm{dim}}V =\sum (\mathrm{dim} V_i)v^{-i}\in \mathbb{Z}[v,v^{-1}].$$

 We define the graded rank of a finitely generated $\mathbb{Z}$-graded $R$-module $M$ as follows
$$\underline{\mathrm{rk}}M = \underline{\mathrm{dim}}(M/MR_+)\in \mathbb{Z}[v,v^{-1}], $$
where $R_+$ is the ideal of $R$ generated by the homogeneous elements of non zero degree.
We have $\underline{\mathrm{dim}}(V(1))=v(\underline{\mathrm{dim}}V)$ and
$\underline{\mathrm{rk}}(M(1))=v(\underline{\mathrm{rk}}M).$ We define 
$\underline{\overline{\mathrm{rk}}}M$ as the image of $\underline{\mathrm{rk}}M$ under
 $v\mapsto v^{-1}$.

Now that we have introduced these notations we can explain Soergel's inverse theorem. Soergel \cite[Theorem 5.3]{So4}  proves that the categorification $\varepsilon: \mathcal{H} \rightarrow \langle \mathcal{B}\rangle$ admits an inverse $\eta : \langle \mathcal{B}\rangle \rightarrow \mathcal{H}$ given by
\begin{equation}\label{eta}  \langle B\rangle \to \sum_{x\in W}
 \underline{\overline{\mathrm{rk}}}\mathrm{Hom}(B, R_x)T_x.\end{equation}

Recall that $R_x$ was defined in section \ref{ca}.
The following notations will be necessary to understand the graded rank formula for the Hom spaces given by Soergel.

%$\bullet$ If $M$ is a left graded $R-$module isomorphic to $P\cdot R,$ with $P\in \mathbb{N}[q,q^{-1}],$ we define the \emph{graded rank} $$\underline{\overline{\mathrm{rk}}}(M)=P.$$

$\bullet$ Let $X$ be a set of homogeneous elements of graded vector spaces. We define the \emph{degree of $X$} $$d(X)=\sum_{x\in X}q^{\mathrm{deg}(x)/2}.$$

$\bullet$ For every $x\in W,$ the functional $\tau_x :\mathcal{H}\rightarrow \mathbb{Z} [v,v^{-1}]$ is defined by
$$\tau_x\left(\sum_{y\in W}p_y\widetilde{T}_y\right) =p_x \,\,\,\,\,\,\,\,\,\,\,\, (p_y\in \mathbb{Z} [v,v^{-1}] \ \  \forall y\in W).$$
We denote $\tau_{\mathrm{id}}$ simply by $\tau.$

$\bullet$ Finally we define the element $C'_{\underline{s}}=C'_{s_1}\cdots
C'_{s_n}\in \mathcal{H}.$

\subsection{Graded degrees}\label{gd}

  Let $\underline{t}$ be an arbitrary sequence of elements in $\mathcal{S}$. We recall \cite[Lemma 5.6]{Li1} 
\begin{equation}\label{as}d(\mathbb{L}_{\underline{t}}(x))=\tau_x(C'_{\underline{t}}). \end{equation}
We remark that the construction of $\mathbb{T}_{\underline{t}}$ is motivated by (and can be seen as a categorification of)
this formula.
As $\eta$ is the inverse of $\varepsilon$ we obtain 

\begin{equation}\label{uh}
\begin{array}{lll}
\tau_x(C'_{\underline{t}})&=&  \tau_x \circ \eta(\langle B_{\underline{t}}\rangle) \\
  &=&\underline{\overline{\mathrm{rk}}}\mathrm{Hom}(B_{\underline{t}},R_x).
\end{array}
\end{equation}

So we have \begin{equation}\label{grados}d(\mathbb{L}_{\underline{t}}(x))=\underline{\overline{\mathrm{rk}}}\mathrm{Hom}(B_{\underline{t}},R_x),\end{equation}
and in particular \begin{equation}\label{rk}d(\mathbb{L}_{\underline{t}}(\mathrm{id}))=\underline{\overline{\mathrm{rk}}}\mathrm{Hom}(B_{\underline{t}},R)\end{equation} as Theorem \ref{LLB} asserts. This formula says that the double leaves basis has the correct degrees. We will prove that the double leaves basis for $\mathrm{Hom}(B_{\underline{s}}, B_{\underline{r}})$ as defined in \ref{LL} also has the correct degrees, i.e. we will prove the formula
\begin{equation}\label{cro}d(\mathbb{L}_{\underline{r}}^a \cdot \mathbb{L}_{\underline{s}})=\underline{\overline{\mathrm{rk}}}\mathrm{Hom}(B_{\underline{s}},B_{\underline{r}}).\end{equation}

%There exist  (see \cite[Corollaire 4.2]{Li1}) an isomorphism of graded left $R$-modules $$\mathrm{Hom}(B_{\underline{s}} ,R) \cong \tau((1+T_{\underline{s}}))\cdot R. $$

Let $\underline{r}^{\mathrm{op}}=(r_p,\ldots,r_2,r_1)$. Using repeatedly the adjunction isomorphism (\ref{a}) we obtain
\begin{equation}\label{g}\mathrm{Hom}(B_{\underline{s}},B_{\underline{r}})\cong \mathrm{Hom}(B_{\underline{s}}B_{\underline{r}^{\mathrm{op}}},R).\end{equation}
Using equations  (\ref{as}), (\ref{rk}) and (\ref{g}) we obtain

$$\label{gh}\underline{\overline{\mathrm{rk}}}\mathrm{Hom}(B_{\underline{s}},B_{\underline{r}}) = \tau(C'_{\underline{s}}C'_{\underline{r}^{\mathrm{op}}}).$$

For any sequence $\underline{u}$ of elements of $\mathcal{S}$ and every $y\in W$ we define the polynomials $p_y^{\underline{u}}$ by the formulas

$$C'_{\underline{u}}=\sum_yp_y^{\underline{u}}\widetilde{T}_y$$

It is easy to check that $p_y^{\underline{u}^{\mathrm{op}}}=p^{\underline{u}}_{y^{-1}}.$ Using the following equation  (see \cite[proposition 8.1.1]{GP})

\begin{equation} \tau(\widetilde{T}_x\widetilde{T}_{y^{-1}})=  \delta_{x,y}\end{equation}
 
 we conclude that 
 $$\tau(C'_{\underline{s}}C'_{\underline{r}^{\mathrm{op}}})=\sum_{x\in W}p_x^{\underline{s}}p_x^{\underline{r}}.$$

 It is easy to see that the degrees of the generating morphisms $f_{sr}, j_s, m_s$ are equal to the degrees of their adjoints, so we conclude that the degree of any morphism between Bott-Samelson bimodules is equal to the degree of its adjoint, so $$d(\mathbb{L}_{\underline{r}}(x)^a)=p_x^{\underline{r}}.$$ We thus obtain $$d(\mathbb{L}_{\underline{r}}(x)^a\cdot \mathbb{L}_{\underline{s}}(x))=p_x^{\underline{s}}p_x^{\underline{r}}$$ 
and finally
\begin{displaymath}
\begin{array}{lll}
d(\mathbb{L}_{\underline{r}}^a\cdot \mathbb{L}_{\underline{s}})&=&\sum_{x\in W}d(\mathbb{L}_{\underline{r}}(x)^a\cdot \mathbb{L}_{\underline{s}}(x)) \\
&=& \sum_{x\in W}p_x^{\underline{s}}p_x^{\underline{r}}.
\end{array}
\end{displaymath}
Thus proving formula (\ref{cro}).

\subsection{An important Lemma}\label{T} Let us suppose that the elements of $\mathbb{L}_{\underline{r}}^a\cdot \mathbb{L}_{\underline{s}}$ are linearly independent for the left action of $R$. Let $M$ be the sub $R$-module of $\mathrm{Hom}(B_{\underline{s}},
B_{\underline{r}})$  generated by the elements of $\mathbb{L}_{\underline{r}}^a\cdot \mathbb{L}_{\underline{s}}$. In every degree the modules $M$ and $\mathrm{Hom}(B_{\underline{s}},
B_{\underline{r}})$ are finite dimensional $k$-vector spaces, and they have the same dimension as we saw in Section \ref{gd}, so they are equal. This being true in every degree we deduce that   $M=\mathrm{Hom}(B_{\underline{s}},
B_{\underline{r}})$, and this would complete the proof of Theorem \ref{LL}. 

So we only need to prove that the elements of $\mathbb{L}_{\underline{r}}^a\cdot \mathbb{L}_{\underline{s}}$ are linearly independent for the left action of $R$.
Before we can do this we have to introduce an order on $\mathbb{L}_{\underline{r}}^a\cdot \mathbb{L}_{\underline{s}}$. The linear independence will follow from a triangularity condition with respect to this order. 

Recall that $n$ is the lenght of $\underline{s}$. To every element $l\in \mathbb{L}_{\underline{s}}$ we associate two elements $\underline{i}=(i_1,\ldots, i_n)$ and $ \underline{j}=(j_1,\ldots, j_n)$ of $\{0,1\}^n$ in the following way. By construction $l$ is a composition of $n$ morphisms, say $l=l_n\circ \cdots \circ l_1.$
For $1\leq k\leq n$ we put $i_k=1$ if in $l_k$ (that is a composition of morphisms of the type $m_s, j_s$ and/or $f_{sr}$) appears a morphism of the type $m_s$ and we put $i_k=0$ otherwise. We put $j_k=1$ if in $l_k$  appears a morphism of the type $j_s$ and we put $j_k=0$ otherwise. 

On the other hand, the leaf $l$ is completely determined by $\underline{s},$ $\underline{i}$ and $ \underline{j},$ (moreover, it is completely determined by $\underline{s}$ and $ \underline{i},$) so we denote $l=f^{\underline{j}}_{\underline{i}}.$ The following lemma is key for defining the mentioned order.

%\begin{lem}\label{IJ} Let  $\underline{i}\neq \underline{i'}\in \{0,1\}^m\times \{0,1\}^m$. If  $f^{\underline{j}}_{\underline{i}}\in \mathbb{L}_{\underline{s}}(x)$ and  $f^{\underline{j}}_{\underline{i'}}\in \mathbb{L}_{\underline{s}}(x')$, then $l(x)\neq l(x').$\end{lem}
%\begin{proof}

%Let us prove this lemma by contradiction: we will suppose that $l(x)=l(x')$. For $1\leq p\leq m$, let $y_p$ and $y'_p$ be such that $f^{\underline{j}}_{\underline{i}}(p)\in A_p(\underline{s},y_p)$ and $f^{\underline{j}}_{\underline{i'}}(p)\in A_p(\underline{s},y'_p)$. By construction we have
%$$ l(y_{p-1}s_p)=l(y_{p-1})-1 \iff l(y'_{p-1}s_p)=l(y'_{p-1})-1. $$ 
%We have supposed that $l(y_m)=l(y'_m),$ we will prove by descendent induction that $l(y_p)=l(y'_p)$ for all $1\leq p\leq m$ and this will yield the contradiction we are looking for.

%Let $1\leq p\leq m$. Let us suppose that $l(y_p)=l(y'_p)$, we will prove that $l(y_{p-1})=l(y'_{p-1}).$ By construction of $A_m(\underline{s})$ we see that $y_p=y_{p-1}$ or $y_p=y_{p-1}s_p.$ Let us suppose $l(y_{p-1})\neq l(y'_{p-1}),$ then we can suppose without loss of generality that $y_p=y_{p-1}s_p$ and $y'_{p}=y'_{p-1}.$ Let us consider the two possible cases:

\begin{lem}\label{IJ}
Let  $(\underline{i}, \underline{i}')\in \{0,1\}^n\times \{0,1\}^n$ with $\underline{i}\neq \underline{i}'$. If  $f^{\underline{j}}_{\underline{i}}\in \mathbb{L}_{\underline{s}}(x)$ and  $f^{\underline{j}}_{\underline{i}'}\in \mathbb{L}_{\underline{s}}(x')$, then $x\neq x'.$
\end{lem}
\begin{proof}

Let us prove this lemma by contradiction. We will suppose that $x=x'$. Let $1\leq p\leq n$. If the node of depth $p-1$ corresponding to the leaf  $f^{\underline{j}}_{\underline{i}}$  is colored by the bimodule $(B_{t_1}  \cdots B_{t_{a}})(B_{s_{p}} \cdots B_{s_n}),$ then we define $$y_{p-1}=t_1t_2\cdots t_a\in W,$$
and $y_n=x$. In a similar way we define, for $1\leq p\leq n$,  the element $y'_{p-1}\in W$ associated to the leaf  $f^{\underline{j}}_{\underline{i}'}$ and $y'_n=x'.$ By construction we have
$$ l(y_{p-1}s_p)=l(y_{p-1})-1 \iff l(y'_{p-1}s_p)=l(y'_{p-1})-1. $$ 
We have supposed that $y_n=y'_n.$ We will prove by descendent induction that $y_p=y'_p$ for all $1\leq p\leq n$ and this will imply the equality $\underline{i}= \underline{i}',$ yielding the contradiction we are looking for.

Let $1\leq p\leq n$. Let us suppose that $y_p=y'_p$, we will prove by contradiction that $y_{p-1}=y'_{p-1},$ so we will suppose $y_{p-1}\neq y'_{p-1}.$ By construction of $\mathbb{T}_{\underline{s}}$ we know that $y_p=y_{p-1}$ or $y_p=y_{p-1}s_p,$  then we can suppose without loss of generality that $y_p=y_{p-1}s_p$ and $y'_{p}=y'_{p-1}.$ Let us consider the two possible cases:

\begin{itemize}
\item \textbf{Case 1}: $l(y_{p-1}s_p)=l(y_{p-1})-1. $ As we have seen, this implies  $ l(y'_{p-1}s_p)=l(y'_{p-1})-1.$ But this yields to a contradiction:

\begin{displaymath}
\begin{array}{lll}
l(y_{p-1})-1&=&l(y_p)   \\
  &=&l(y'_{p-1})\\
  &=& l(y'_{p-1}s_p)+1\\
 &=& l(y_ps_p)+1\\
&=&l(y_{p-1})+1 .
\end{array}
\end{displaymath}
\item \textbf{Case 2}: $l(y_{p-1}s_p)=l(y_{p-1})+1. $ This implies  $ l(y'_{p-1}s_p)=l(y'_{p-1})+1.$ This also yields to a contradiction:

\begin{displaymath}
\begin{array}{lll}
l(y_{p-1})&=&l(y_ps_p)   \\
  &=&l(y'_{p-1}s_p)\\
  &=& l(y'_{p-1})+1\\
 &=& l(y_p)+1\\
&=&l(y_{p-1}s_p)+1\\ 
&=&l(y_{p-1})+2.
\end{array}
\end{displaymath}
\end{itemize}
This finishes the proof of Lemma \ref{IJ}.
\end{proof}

\subsection{The order} Consider the element $(f^{\underline{l}}_{\underline{i}})^a\cdot f^{\underline{j}}_{\underline{k}}\in \mathbb{L}_{\underline{r}}(x)^a\cdot \mathbb{L}_{\underline{s}}(x)$. Recall that the sequences $\underline{r}$ and $\underline{s}$ have already been fixed.  As the couple $(\underline{i},\underline{r})$ determine $\underline{l}$ and $x$ and the  triplet $(\underline{j},\underline{s},x)$ determine $\underline{k}$ by Lemma \ref{IJ}, we conclude that the morphism $(f^{\underline{l}}_{\underline{i}})^a\cdot f^{\underline{j}}_{\underline{k}}$ is completely determined by $(\underline{i},\underline{j}).$ We will call it $f_{\underline{i}}f^{\underline{j}}.$

 Now we can introduce a total order $\trianglelefteq$ in $\{0,1\}^{n}\times \{0,1\}^n.$ The symbol $\leq$ will be used for the lexicographical order in $\{0,1\}^{n}.$  Let $\underline{i}=(i_1,\ldots i_n), \underline{i}'=(i'_1,\ldots i'_n), \underline{j}=(j_1,\ldots j_n)$ and $ \underline{j}'=(j'_1,\ldots j'_n)$. The order is defined as follows.
 \begin{itemize}
    \item if $\underline{j}<\underline{j}' $ then $ (\underline{i},\underline{j})\vartriangleleft (\underline{i}',\underline{j}')$
    \item if $\underline{j}=\underline{j}'$ then 
    \begin{itemize}
    \item  if $\sum_{p=1}^ni_p<\sum_{p=1}^ni'_p$ then $ (\underline{i},\underline{j})\vartriangleleft (\underline{i}',\underline{j}')$
    \item if $\sum_{p=1}^ni_p=\sum_{p=1}^ni'_p$ and $\underline{i}<\underline{i}' $ then $ (\underline{i},\underline{j})\vartriangleleft (\underline{i}',\underline{j}')$
    \end{itemize}
\end{itemize}
 
 \begin{remark}
In the notation of \cite{Li1}, if  $\underline{j}=\underline{j}'$, then $ (\underline{i},\underline{j})\trianglelefteq (\underline{i}',\underline{j}')$ if and only if $ \underline{i}\preceq \underline{i'}.$
\end{remark} 
 
 This order in $\{0,1\}^{n}\times \{0,1\}^n$ induces an order  in  $\mathbb{L}_{\underline{r}}^a\cdot \mathbb{L}_{\underline{s}}$:   $$ f_{\underline{i}}f^{\underline{j}} \trianglelefteq f_{\underline{i}'}f^{\underline{j}'}\Longleftrightarrow (\underline{i},\underline{j}) \trianglelefteq (\underline{i}',\underline{j}').$$

In the next paragraphs we will see why is this order particularly interesting for the double leaves basis.

\subsection{Linear independence}\label{opp}\label{li}
 We recall that if $\underline{t}=(t_1,\ldots, t_a)\in \mathcal{S}^a,$  then the opposite sequence is $\underline{t}^{\mathrm{op}}=(t_a,\ldots, t_1)\in \mathcal{S}^a.$
Let us define $\alpha_s=p_s\circ \epsilon_s$. The category $\mathcal{B}$ of Soergel bimodules is generated as tensor category by the morphisms in the set $$\{m_s,j_s, \alpha_s, f_{sr}\}_{s\neq r\in \mathcal{S}}$$ because of Theorem \ref{LLB} and adjunction (\ref{a}). So to any morphism between Bott-Samelson bimodules $f\in \mathrm{Hom}(B_{\underline{t}},B_{\underline{q}})$ there is an  associated morphism that we denote $f^{\mathrm{op}}\in \mathrm{Hom}(B_{\underline{t}^{\mathrm{op}}},B_{\underline{q}^{\mathrm{op}}})$. Moreover, this gives  a bijection between the spaces $\mathrm{Hom}(B_{\underline{t}},B_{\underline{q}})$ and $\mathrm{Hom}(B_{\underline{t}^{\mathrm{op}}},B_{\underline{q}^{\mathrm{op}}})$.

Recall that in the construction of the tree $\mathbb{T}_{\underline{s}}$ we fixed for every element $x\in W$ a reduced expression $\underline{x}=(x_1,\ldots, x_a)$ of $x$.  We define the following morphism in $\mathrm{Hom}(B_{\underline{x}}B_{\underline{x}^{\mathrm{op}}},R):$ $$\rho_{x}=(m_{x_1}\circ j_{x_1})\circ (\mathrm{id}^1\otimes (m_{x_2}\circ j_{x_2})\otimes \mathrm{id}^1)\circ\cdots\circ(\mathrm{id}^{a-1}\otimes (m_{x_a}\circ j_{x_a})\otimes \mathrm{id}^{a-1}).$$

 In the isomorphism (\ref{g}) the subset $\mathbb{L}_{\underline{r}}^a\cdot \mathbb{L}_{\underline{s}}\subseteq \mathrm{Hom}(B_{\underline{s}}, B_{\underline{r}})$ corresponds to   $$\mathbb{L}_{\underline{s}}\bullet  \mathbb{L}_{\underline{r}}^{\mathrm{op}} \subseteq \mathrm{Hom}(B_{\underline{s}}B_{\underline{r}^{\mathrm{op}}},R).$$
where $\bullet$ is defined as follows.  If $f\in\mathbb{L}_{\underline{s}}(x)$ and $g \in \mathbb{L}_{\underline{r}}(y)$,  then 
$$f\bullet g^{\mathrm{op}}=
\begin{cases}
\hspace*{1.15cm} 0\hspace*{1.2cm} \mathrm{if}\ x\neq y\\ 
\rho_{x} \circ (f\otimes g^{\mathrm{op}}) \ \ \ \mathrm{if}\ x=y
\end{cases}$$

So it is enough to prove that $ \mathbb{L}_{\underline{s}}\bullet \mathbb{L}_{\underline{r}}^{\mathrm{op}}$ is a linearly independent set.  For any sequence $(u_1, \ldots, u_p)\in\mathcal{S}^p$, let us denote  $$1^{\otimes}=1\otimes 1\otimes \cdots\otimes 1\in B_{u_1}\cdots B_{u_p}$$ and $$x^{\otimes}=1\otimes x_{u_1}\otimes \cdots\otimes x_{u_p}\in B_{u_1}\cdots B_{u_p}.$$
Let us recall that $R_+$ is the ideal of $R$ generated by the homogeneous elements of non zero degree. An element in $B_{u_1}\cdots B_{u_p}$ is called superior if it belongs to the set $R_+B_{u_1}\cdots B_{u_p}$ and it is called normalsup if it belongs to the set $x^{\otimes}+R_+B_{u_1}\cdots B_{u_p}.$ For $s\in\mathcal{S}$,  let us denote $x_s^0=1$ and $x_s^1=x_s.$ If $\underline{j}=(j_1,\ldots, j_n)\in \{0,1\}^n$, we put $$x^{\underline{j}}=x_{s_1}^{j_1}\otimes\cdots\otimes x_{s_n}^{j_n}\otimes 1\in B_{s_1}\cdots B_{s_n}.$$
By construction of $f^{\underline{j}}_{\underline{k}}\in \mathbb{L}_{\underline{s}}$ we have that $f^{\underline{j}}_{\underline{k}}(x^{\underline{j}})=1^{\otimes}$ and $f^{\underline{j}}_{\underline{k}}(x^{\underline{j}'})=0$ if $\underline{j}'<\underline{j}.$ We remark that this is independent of $\underline{k}$. On the other hand,  if $$x_{\underline{i}}=1\otimes x_{r_1}^{i_1}\otimes\cdots\otimes x_{r_p}^{i_p}\in B_{r_1}\cdots B_{r_p},$$
then $f_{\underline{i}}^{\underline{k}}(x_{\underline{i}})$ is normalsup \cite[Lemme 5.10]{Li1} and $f_{\underline{i}}^{\underline{k}}(x_{\underline{i}'})$ is a superior element if $\underline{i}'\prec \underline{i}.$ Again, this is independent of $\underline{k}.$

Let us define $$x_{\underline{i}}^{\mathrm{op}}=x_{r_p}^{i_p}\otimes\cdots \otimes x_{r_1}^{i_1}\otimes 1\in B_{r_p}\cdots B_{r_1}.$$
 If $f_{\underline{i}}f^{\underline{j}}\in \mathbb{L}_{\underline{r}}^a\cdot \mathbb{L}_{\underline{s}}$, by explicit calculation (using the fact that $f_{sr}(1^{\otimes})=1^{\otimes}$ and some simple degree arguments) we obtain the formula
\begin{equation}\label{uni}f^{\underline{j}}\bullet f^{\mathrm{op}}_{\underline{i}}(x^{\underline{j}'}\otimes x^{\mathrm{op}}_{\underline{i}'})=  \begin{cases} 1
 \text{ if } (\underline{i}',\underline{j}')= (\underline{i},\underline{j}) \\
 0 \text{ if } (\underline{i}',\underline{j}')\vartriangleleft (\underline{i},\underline{j})
 \end{cases}\end{equation}
 
 This unitriangularity formula proves that $ \mathbb{L}_{\underline{s}}\bullet \mathbb{L}_{\underline{r}}^{\mathrm{op}}$ is a linearly independent set and thus we finish the proof of Theorem \ref{LL}. $\hfill \Box$

\begin{cor}\label{p} The set $\mathbb{L}_{\underline{r}}^a\cdot p\mathbb{L}_{\underline{s}}$  is  a basis of  $\mathrm{Hom}(B_{\underline{s}}, B_{\underline{r}}).$
\end{cor}
\begin{proof}
The graded degrees of $\mathbb{L}_{\underline{r}}^a\cdot p\mathbb{L}_{\underline{s}}$ are the same as the graded degrees of $\mathbb{L}_{\underline{r}}^a\cdot \mathbb{L}_{\underline{s}}$ because $p_{\underline{t}}$ is a degree zero morphism for every $\underline{t}$ reduced expression of an element of $W$.

By degree reasons, using Proposition \ref{grande} (3) we can see that the element $1^{\otimes}\in B_{\underline{t}}$ is in the image of  $p_{\underline{t}}$. So we have a unitriangularity  formula analogous to (\ref{uni}) that gives us the linear independence. 
\end{proof}

\section{Standard Leaves basis}\label{SLB}

In the rest of  this section we prove a result that is not needed for the rest of the paper (the reader only interested in Lusztig's Conjecture can skip this section), but it is a simple corollary of Lemma \ref{IJ}. We expect that it will help in  the understanding of Kazhdan-Lusztig polynomials because it gives an important step in the calculation of  the character of any Soergel bimodule that is given as the image of a projector of a Bott-Samelson in the spirit of \cite{Li4} or \cite{El} (see the formula for $\eta$ in section \ref{Some notations}).

To generalize Theorem \ref{LLB} we need to introduce a new morphism. Consider the $(R,R)-$bimodule morphism $\beta: B_s \rightarrow R_s$ defined by $\beta(p\otimes q)=ps(q)$ for $p,q\in R.$ If $x,y\in W$ we have $R_xR_y\cong R_{xy}$. If $\underline{x}=(x_1,\ldots ,x_a),$  we can define an $(R,R)-$bimodule morphism that we also denote $\beta:B_{\underline{x}}\rightarrow R_x.$

For a sequence $\underline{s}$ of elements in $\mathcal{S}$ we define the set $$\mathbb{L}_{\underline{s}}^{\beta}=\{\beta\circ l \,\vert\, l\in \mathbb{L}_{\underline{s}}\}\subseteq \coprod_{x\in W}\mathrm{Hom}(B_{\underline{s}},R_x). $$
We define $\mathbb{L}_{\underline{s}}^{\beta}(x)$ as the subset of $\mathbb{L}_{\underline{s}}^{\beta}$ consisting of the elements belonging to $\mathrm{Hom}(B_{\underline{s}},R_x)$. We call the set $\mathbb{L}_{\underline{s}}^{\beta}$ the \emph{Standard leaves basis}, and the following generalization of Theorem \ref{LLB} explains this name.

\begin{prop}\label{Lb}
The set $\mathbb{L}_{\underline{s}}^{\beta}(x)$ is a basis of $\mathrm{Hom}(B_{\underline{s}},R_x)$ as a right $R$-module.
\end{prop}

\begin{proof}
Because of the  equation (\ref{grados}) and using a similar reasoning as in Section \ref{T} we only need to prove that the elements of $\mathbb{L}_{\underline{s}}^{\beta}(x)$ are linearly independent with respect to the right $R$-action.

From section \ref{li} we know that $\beta\circ f^{\underline{j}}_{\underline{i}}(x^{\underline{j}})=1$  and that   $\beta\circ f^{\underline{j}}_{\underline{i}}(x^{\underline{j}'})=0,$ if $\underline{j}$ is bigger than $\underline{j}'$ in the lexicographical order. Then Lemma \ref{IJ} allows us to conclude the proof by a triangularity argument.\end{proof}

\section{Soergel's affine $p$-conjecture is equivalent to Fiebig's conjecture}\label{eq1}

In this section $W$ is a Weyl group and $\widehat{W}$ its associated affine Weyl group.
Soergel proves  \cite{So3}  that Soergel's conjecture for Weyl groups is equivalent to a part of Lusztig's conjecture (weights near the Steinberg weight). We will prove in this section  that a version of Soergel's conjecture (that we called Soergel's afine $p$-conjecture) implies the full Lusztig conjecture.

 More precisely, it is in some sense implicit in the work of Fiebig that Soergel's affine $p$-conjecture  is equivalent to Fiebig's conjecture (we will explain in detail Soergel's affine $p$-conjecture, Fiebig's conjecture and their equivalence in the Section \ref{SF}). Fiebig proves \cite{Fi3} that Fiebig's conjecture implies Lusztig's conjecture.

\subsection{The category $\mathcal{P}^{\circ}$}\label{P0}
In the following we use the same notations as in sections 2.2 and 2.3 of \cite{Fi2}.
Let us suppose that $p$ is  bigger than the Coxeter number of $W$.  %We recall that $\widehat{W}^{\circ}$ is the finite subset of $\widehat{W}$ consisting of the elements that are lesser or equal (in the Bruhat order) than the longest element in the set of antidominant restricted elements.
Define $$\hCZ:= \mathcal{Z}(\widehat{W}^{\circ})$$

Let us define $\widehat{\mathcal{B}}^{\circ}_p$ as the full subcategory of $\mathcal{B}(\widehat{W}, \mathbb{F}_p)$ consisting of objects that are direct sums of objects in the set $\{B_w\}_{w\in \widehat{W}^{\circ}}.$

We define (following \cite[Proof of Theorem 4.3]{Fi3}) a functor $G:\widehat{\mathcal{B}}^{\circ}_p\rightarrow \hCZ$-mod by the rule 
$$
\widehat{\mathcal{B}}^{\circ}_p\ni M\mapsto G(M):= \left\{ (m_x)\in \bigoplus_{x\in \widehat{W}^{\circ}}M^x\left| \, 
\begin{matrix} 
\rho_{x,tx}(m_x)= \rho_{tx,x}(m_{tx})\\
 \text{ for all $t\in \mathcal{T}^{\circ},$ $x, tx\in \widehat{W}^{\circ}$}
\end{matrix}
\right.\right\}.
$$
The bimodule $M^x$ and the morphisms $\rho_{x,y}$ were defined in section \ref{ca}.
$G(B)$ is a $\hCZ$-module by pointwise multiplication. We define $\mathcal{P}^{\circ}$ as the essential image of $G$. For $\mathcal{M}\in \mathcal{P}^{\circ}$ define $\mathrm{supp}\mathcal{M}:=\{w\in \widehat{W}\,\vert \, \mathcal{M}^w\neq 0\}.$ 

%The following lemma is a consequence of the definitions. 
%\begin{lem}\label{lemma-stalks} The costalk $\CM^x$  is the biggest submodule of
 % $\CM$ on which $(z_y)\in\hCZ$ acts as multiplication with
 % $z_x$. Analogously, the stalk $\CM^x$ is the biggest quotient of
 % $\CZ$ on which $(z_y)$ acts as multiplication with $z_x$.
%\end{lem} 

\subsection{Fiebig's conjecture}\label{SF}
The following is Theorem 6.1 of \cite{Fi2}, that will be needed to state Fiebig's conjecture.
%What we call in this paper ``Fiebig's conjecture" is Conjecture 2.10 in \cite{Fi4}.  We now define the category $V$... 
\begin{deth}\label{f}
For all $w\in\widehat{W}^{\circ}$ there exist an object $\mathcal{B}(w)\in \mathcal{P}^{\circ}$, unique up to isomorphism, with the following properties: 
\begin{enumerate}
\item $\mathcal{B}(w)$ is indecomposable  in $\mathcal{P}^{\circ}$.
\item $\mathrm{supp}\mathcal{B}(w)\subseteq \  \leq w$ and $\mathcal{B}(w)^w\cong R(l(w))$
\end{enumerate}
\end{deth}

We note that  we call $\mathcal{B}(w)$ what is called $\mathcal{B}(w)(l(w))$ in \cite{Fi5}. Let us recall Fiebig's conjecture.

\begin{conj}[Fiebig] $C'_w=\sum_{x}\underline{\mathrm{rk}}\mathcal{B}(w)^xT_x\ \  for  \  all\ w\in \widehat{W}^{\circ}.$
\end{conj}

We note that because of \cite[Thm 2.11]{Fi4} (using the equivalence \cite[Prop. 3.4]{Fi4}) we have that Fiebig's conjecture implies Lusztig's conjecture for the dual group $G_{\overline{k}}^{\vee}.$

By \cite[Corollary 3.10 (1) (b)]{Fi4} we have that $\underline{\mathrm{rk}}\mathcal{B}(w)^x=\underline{\overline{\mathrm{rk}}}\mathcal{B}(w)_x.$
On the one hand this means that part (2) of Definition/Theorem \ref{f} is equivalent to 
\begin{itemize}
\item[(2')] $\mathrm{supp}\mathcal{B}(w)\subseteq \  \leq w \ \mathrm{and}\   \mathcal{B}(w)_w\cong R(-l(w))$
\end{itemize}
 and on the other hand  we can restate Fiebig's conjecture:
\begin{equation}\label{F}
C'_w=\sum_{x}\underline{\overline{\mathrm{rk}}}\mathcal{B}(w)_xT_x\ \  for  \  all\ w\in \widehat{W}^{\circ}
\end{equation}

\subsection{Equivalence}\label{equiv}

Let $\eta:\langle \mathcal{B}\rangle\rightarrow \mathcal{H} $ be the inverse of the isomorphism of rings  $\varepsilon: \mathcal{H}\rightarrow \langle \mathcal{B}\rangle $ defined in Section \ref{So}.
Soergel \cite[Prop 5.9]{So4} proves that  $\eta=d\circ h_{\nabla},$ (with $d$ as defined in Section \ref{HA}) and $h_{\nabla}$ defined by the following formula (see \cite[Korollar 5.16]{So4})
$$h_{\nabla}(B)=\sum_{x\in \widehat{W}}\underline{\overline{\mathrm{rk}}}\mathrm{Hom}(R_x,B)T_x.$$

An equivalence between additive categories sends indecomposable objets to indecomposable objects, so looking at Proposition \ref{grande} (1) and Definition/Theorem \ref{f} we can see that  the functor $G,$ that restricts \cite[Thm 6.3]{Fi2}  to an equivalence between $\widehat{\mathcal{B}}^{\circ}_p$ and $\mathcal{P}^{\circ},$ sends the bimodule $B_w$  to the $\mathcal{Z}^{\circ}-$module  $\mathcal{B}(w)$.

It is easy to verify (see Section \ref{ca}) that for all $M\in \mathcal{B}$ we have  \begin{equation}\label{ho}\mathrm{Hom}(R_x,M)\cong M_x\end{equation}
(the evaluation at $1_x\in R_x$ gives the isomorphism). Moreover $M_x$ is the biggest submodule of $M$ in which $r\in R\otimes R$ acts as multiplication by $r_x\in R_x.$ By \cite[Lemma 3.3]{Fi4} we have that for $\mathcal{M}\in \mathcal{Z}^{\circ}-$mod, $\mathcal{M}_x$ is the biggest submodule of $\mathcal{M}$ on which $(z_y)\in \mathcal{Z}^{\circ}$ acts as multiplication by $z_x,$ so we have an isomorphism
\begin{equation}\label{B}(B_w)_x\cong \mathcal{B}(w)_x\end{equation}
as graded right $R$-modules.

An equivalent formulation of Soergel's affine $p$-conjecture  is that for all $w\in \widehat{W}^{\circ},$ we have $\eta(\langle B_w\rangle)=C'_w.$
As $C'_w$ is self-dual, using equation (\ref{ho}) we see that Soergel's affine $p$-conjecture is equivalent to 
\begin{equation}\label{S}
C'_w=\sum_{x}\underline{\overline{\mathrm{rk}}}(B_w)_xT_x\ \  \mathrm{for}  \  \mathrm{all}\ w\in \widehat{W}^{\circ}
\end{equation}

So the isomorphism (\ref{B}) imply that Fiebig's conjecture (\ref{F}) is equivalent to Soergel's affine $p-$conjecture (\ref{S}).
 $\hfill \Box$

\newpage
\section*{Appendix}
Let $R$ be a root system, $B$ a basis of $R$, $A=(a_{st})_{s,t\in \mathcal{S}}$ its Cartan matrix  and $W$ its Weyl group.  Consider the free $\mathbb{Z}$-module $E_{\mathbb{Z}}=\oplus_{s\in \mathcal{S}} \mathbb{Z}\alpha_s^{\lor}$ with the action of the Weyl group defined linearly by the equation $$ t(\alpha_s^{\lor})= \alpha_s^{\lor}- a_{st} \alpha_t^{\lor} \hspace{.5cm}\mathrm{for\ all\ }s,t\in \mathcal{S}.$$ 

We recall that the \emph{Cartan matrix representation} over the field $k$  is the representation $E_k=E_{\mathbb{Z}}\otimes_{\mathbb{Z}}k$. In this Appendix we prove
%$\mathfrak{h}=\oplus_{s\in \mathcal{S}} k\alpha_s^{\lor}$, where we define $\alpha_s\in \mathfrak{h}^*$ by 

\begin{prop}\label{anex}
The Cartan matrix representation $E_k$ of a Weyl group is reflection faithful if $k$ is a field of characteristic $p\neq 2, 3.$ 
\end{prop}

\begin{proof}

It is easy to see that if the group $W$ acts in the space $L$ and if  $x,y\in W$ then $yL^xy^{-1}\cong L^{yxy^{-1}},$ so we just need to study representatives in conjugacy classes.  
To prove that a representation $E$ is reflection faithful, we need to prove the three following points.

\begin{itemize}
\item[(i)] If $w$ is a reflection, then $\mathrm{codim}_EE^w=1$
\item[(ii)]  If $w$ is not a reflection, then $\mathrm{codim}_EE^w\neq1$
\item[(iii)]  If $w\neq 1$, then $\mathrm{codim}_EE^w\neq 0$ (it is a faithful representation)
\end{itemize}

We start by studying an important case, type A.  The symmetric group $S_n$ acts in $V=k^n$ by permuting the components. The Cartan matrix representation of $S_n$ is equivalent to the following subrepresentation of $k^n$ $$E=\{(a_1,\ldots, a_n)\in k^n\, \vert \, \sum_{i=1}^na_i=0\}.$$

\begin{lem}\label{achi}The representation $E$ is reflection faitfhul if $p\neq 2, 3$.
\end{lem}
\begin{proof}
Any element in the symmetric group $S_n$ is conjugate to an element $w$ of the form 
$$w=\Big(123\cdots \lambda_1\Big)\Big((\lambda_1+1)(\lambda_2+1)\cdots  (\lambda_1+\lambda_2)\Big)\cdots \left((\sum_{i=1}^{m-1}\lambda_i+1)\cdots (\sum_{i=1}^{m}\lambda_i)\right) $$

with $\lambda_1\geq \lambda_2\geq \cdots \lambda_m.$ It is easy to see that
\begin{equation}\label{q}V^w=\{(\underbrace{a, a,\ldots, a}_{\lambda_1}, \underbrace{b, b,\ldots, b}_{\lambda_2}, \ldots, \underbrace{c, c,\ldots, c}_{\lambda_m})\, \vert \, a, b,\ldots, c\in k\}\end{equation}
For this element $w$ let us define the number $h(w):=n-m$. By equation (\ref{q}) we have that $\mathrm{codim}_VV^w=h(w)$

\begin{equation}\mathrm{codim}_EE^w=  \begin{cases} h(w)\hspace{.68cm}
 \ \text{ if } V^w\nsubseteq E \\
 h(w)-1 \ \text{ if }V^w \subseteq E
 \end{cases}\end{equation}

It is clear that $V^w \subseteq E$ if and only if $p\, \vert\, \lambda_i$ for all $1\leq i \leq m.$
 \begin{itemize}
\item[(i)] In the symmetric group, reflections are transpositions. Let $w$ be a transposition. Then $\lambda_1=2$ and $\lambda_i=1$ for $2\leq i\leq m,$ so $h(w)=1.$ We just need to prove that if $p\neq 2$ then $V^w\nsubseteq E $ in order to prove point (i). If $n\geq 3$ this is clear because there is some $\lambda_i=1.$ If $n=2$, it is also clear because $p\neq 2$. 
\item[(ii)]  To prove point (ii) let us consider an element $w$ that is not a transposition,   $w\neq 1$ and such that $\mathrm{codim}_EE^w=1$. We will get a contradiction if $p\neq 2$. It is clear that $h(w)\geq 2$, so $h(w)= 2$ and $V^w \subseteq E.$  The equation $h(w)= 2$ gives us two options: \begin{enumerate}
\item $\lambda_1=\lambda_2=2$ and $\lambda_i=1$ for $3\leq i\leq m$.
\item   $\lambda_1=3$ and $\lambda_i=1$ for $2\leq i\leq m.$
\end{enumerate}
As $V^w \subseteq E,$ option (1) forces $n$ to be $4$ and $p$ to be $2$ and 
option (2) forces $n$ to be $3$ and $p$ to be $3$. 

\item[(iii)] Let $w\neq 1,$ we will suppose that $\mathrm{codim}_EE^w=0$ and we will get a contradiction  if $p\neq 2$.  We have that $w\neq 1$ implies $\mathrm{codim}_VV^w = h(w) \geq 1,$ so $h(w)=1$ and $V^w \subseteq E.$ But $h(w)=1$ if and only if $w$ is a transposition and we have already seen this in case (i). 
\end{itemize}

\end{proof}

\begin{remark} 
It is clear looking at the proof of Lemma \ref{achi} that if $n\geq 5,$ then $E$ is reflection faithful independently of $p$. 
\end{remark}

Now we go to the case of a Weyl group $W$ not in type A. We call $E_k$ the corresponding Cartan matrix representation.  Recall that for all $s \in \mathcal{S},$  the fundamental weight $\varpi_{s}$ belongs to $E_{\mathbb{Q}},$ and 
there is an integer $\lambda_{s}\in \mathbb{Z}$ such that $\lambda_{s}\varpi_{s}\in E_{\mathbb{Z}}.$ This integer, in type other than A, is $2$ or $3$.

The proof of (i) comes from the following formula in $E_{\mathbb{Q}}$, valid for all $s, t \in \mathcal{S}$

$$s_{\beta}\varpi_{\alpha}=\varpi_{\alpha}-\delta_{\alpha\beta}\beta$$

$$t(\varpi_{s})=\varpi_{_s}-\delta_{st}\alpha_t^{\lor}$$

where $\delta_{\alpha\beta}$ is the Kronecker symbol. This implies (i), unless that the reduction modulo $p$ of the Cartan Matrix is not invertible (this would mean that the images of the fundamental weights are not anymore a $k-$basis of $E_k$), and this does not happen if $p\neq 2$, so this completes the proof of (i). 

To prove (ii) and (iii) we need the following lemma. 
\begin{lem} Let $T_{\mathbb{Z}}:E_{\mathbb{Z}}\rightarrow E_{\mathbb{Z}}$ be a linear map.  If $f\in \mathbb{Z}[t]$ divides the characteristic polynomial of $T_{\mathbb{R}}:E_{\mathbb{R}}\rightarrow E_{\mathbb{R}}$, $\mathrm{deg}f\geq 2$  and $f(1)=p_0^m$ for $p_0$ a prime and $m$ a natural number, then $T_k:E_k\rightarrow E_k$ does not fix a hyperplane if $p$ (the characteristic of $k$) is different from $p_0$. 
%then $1$ is not a root of the reduction of $f$ modulo $p$.  
\end{lem}
\begin{proof}
It is easy to see, reducing to the case where $f$ is a monomial, that $f\in \mathbb{Z}[t]$ implies that $f(t)=(t-1)h(t)+f(1)$ with $h(t)\in \mathbb{Z}[t]$. This means that in our case $f(t)=(t-1)h(t)+p_0^m.$ So the reduction $\bar{f}$ of $f$ modulo $p$ does not have $1$ as a root, since $p\neq p_0$ by hypothesis. 

The characteristic polynomial of $T_k$ is the reduction modulo $p$ of the characteristic polynomial of $T_{\mathbb{R}}$. We also have that if  $T_k$  fixes a hyperplane then $(t-1)^{n-1}$ divides the characteristic polynomial of $T_k$, with $n$ being the dimension of $E_k.$ But the reduction  $\bar{f}$ of $f$ modulo $p$ divides the characteristic polynomial of $T_k$ and does not have $1$ as a root. By hypothesis  $\mathrm{deg}f\geq 2$, so this completes the proof of the lemma. 
\end{proof}

In \cite{Ca} it is proved that any element $w$ of a Weyl group decomposes $E_{\mathbb{R}}$ as a direct sum of invariant subspaces. The restriction of the action of $w$ to these subspaces has prescribed characteristic polynomials as in \cite[Table 3]{Ca}. It is easy to see by inspection that all of the characteristic polynomials therein (other than in type A) are divisible by a polynomial $f$ such that $f(1)=1, 2, 3$ or $4$. It is almost always the first factor of the factorization given there, only for the admissible diagram of type $D_2$ you need to take the first and second factor. So we are reduced to a direct sum of type A blocks, so Lemma \ref{achi}  finishes the proof of  Proposition \ref{anex}.\end{proof}

\end{document}